\documentclass[12pt,twosided]{article}
\usepackage{graphicx}
\usepackage{amsmath}
\usepackage{amssymb}
\usepackage{amscd}
\usepackage{ifthen}
\usepackage{tikz}\usetikzlibrary{matrix,arrows,decorations.pathmorphing}
\usepackage[colorlinks, linkcolor=blue, citecolor=magenta, urlcolor=cyan]{hyperref}
\usepackage{mathabx}

\def\rank{\mathop{\rm rank}\nolimits}

\newtheorem{theorem}{Theorem}[section]
\newtheorem{proposition}[theorem]{Proposition}
\newtheorem{corollary}[theorem]{Corollary}
\newtheorem{conjecture}[theorem]{Conjecture}
\newtheorem{remark}[theorem]{Remark}

\newtheorem{lem}[theorem]{Lemma}
\newtheorem{examp}[theorem]{Example}
\newtheorem{prop}[theorem]{Proposition}

\newcommand*{\QEDB}{\hfill\ensuremath{\Box}}%

\newenvironment{proof}[1][Proof]{\textbf{#1.} }{\QEDB \vspace{5pt}}

\newcommand\bond[1]{\draw (#1) -- +(1,0)}
\newcommand\tcirc[3]{
	\ifthenelse{\equal{#1}{w}}{\filldraw[fill=white,draw=black] (#2) circle (0.08);}{}%
	\ifthenelse{\equal{#1}{b}}{\filldraw[black] (#2) circle (0.08);}{}%
	\draw (#2) node[above=2pt] {#3};
	}
\newcommand\tcross[2]{
	\draw (#1) node[above=2pt] {#2};
	\draw (#1) ++(-0.12,-0.12)-- +(0.24, 0.24);
	\draw (#1) ++(-0.12, 0.12)-- +(0.24,-0.24);
	}
\newcommand\tstar[2]{
	\draw[color=red] (#1) node {\Large$*$};
	\draw (#1) node[above=2pt] {#2};
	}
\newcommand\tsquare[2]{
		\draw[semithick,color=blue] (#1) ++(-0.15,-0.15) rectangle ++(0.3,0.3);
		\tcross{#1}{#2};
		}
\newcommand\DDnode[3]{
\ifthenelse{\equal{#1}{w}}{\tcirc{w}{#2}{#3}}{}		
\ifthenelse{\equal{#1}{b}}{\tcirc{b}{#2}{#3}}{}		
\ifthenelse{\equal{#1}{x}}{\tcross{#2}{#3}}{}		
\ifthenelse{\equal{#1}{s}}{\tstar{#2}{#3}}{}		
\ifthenelse{\equal{#1}{q}}{\tsquare{#2}{#3}}{}		
}

\newcommand{\comm}[1]{}

\newcommand\C{{\mathbb C}}
\newcommand\PP{{\mathbb P}}
\newcommand\N{{\mathbb N}}

\newcommand\g{\mathfrak{g}}

\newcommand\h{\mathfrak{h}}

\newcommand\op[1]{\mathop{\rm #1}\nolimits}

\newcommand\p{\partial}

\newcommand\R{{\mathbb R}}

\newcommand\Z{{\mathbb Z}}

\newcommand\hol{\mathfrak{hol}}
\def\Re{\mathop{\rm Re}\nolimits}
\def\Im{\mathop{\rm Im}\nolimits}
\def\SU{\mathop{\rm SU}\nolimits}

\def\PSU{\mathop{\rm PSU}\nolimits}

\makeatletter
\def\blfootnote{\xdef\@thefnmark{}\@footnotetext}
\makeatother

\begin{document}

\title{Blow-ups and infinitesimal automorphisms \\
of CR-manifolds}\blfootnote{{\bf Mathematics Subject Classification:} 32V40, 32C05; 32M12, 53C15.}
 \blfootnote{{\bf Keywords:} real hypersurface in complex space, CR-automorphism, holomorphic vector field, submaximal symmetry dimension, parabolic subalgebra, gap phenomenon.}
\author{Boris Kruglikov\footnote{Department of Mathematics and Statistics, UiT the Arctic University of Norway,
Troms\o{} 90-37, Norway. Email: {\tt boris.kruglikov@uit.no}.}}
 \date{}
 \maketitle

 \begin{abstract} \footnotesize 
For a real-analytic connected CR-hypersurface $M$ of CR-dimension $n\ge 1$ having a point 
of Levi-nondegeneracy
the following alternative is demonstrated for its symmetry algebra $\mathfrak{s}=\mathfrak{s}(M)$:
(i) either $\dim\mathfrak{s}=n^2+4n+3$ and $M$ is spherical everywhere; \linebreak
(ii) or $\dim\mathfrak{s}\le n^2+2n+2+\delta_{2,n}$ and in the case of equality $M$ is spherical
and has fixed signature of the Levi form in the complement to its Levi-degeneracy locus.
A version of this result is proved for the Lie group of global automorphisms of $M$.

Explicit examples of CR-hypersurfaces and their infinitesimal and global automorphisms realizing
the bound in (ii) are constructed. We provide many other models with large symmetry using the technique of blow-up,
in particular we realize all maximal parabolic subalgebras of the pseudo-unitary algebras as a symmetry.
 \end{abstract}

\section{Introduction}\label{introduction}

\subsection{Formulation of the problem}

Investigation of symmetry is a classical problem in geometry. For a class $\mathcal{C}$ of manifolds
endowed with particular geometric structures, denote by ${\mathfrak s}(M)$ the Lie algebra 
of vector fields on $M$ preserving the structure (infinitesimal automorphisms). 
It is important to determine the maximal value $D_{\hbox{\tiny\rm max}}$ of the 
{\it symmetry dimension} $\dim{\mathfrak s}(M)$ over all $M\in\mathcal{C}$.

Often the values immediately below $D_{\hbox{\tiny\rm max}}$ are not realizabile as 
$\dim{\mathfrak s}(M)$ for any $M\in{\mathcal C}$, which is known as the {\it gap phenomenon}. 
One then searches for the next realizable value, the {\it submaximal dimension $D_{\hbox{\tiny\rm smax}}$}, thus obtaining the interval $(D_{\hbox{\tiny\rm smax}}, D_{\hbox{\tiny\rm max}})$ called the first gap 
(or lacuna) for the symmetry dimension. 

The first and next gaps were successfully identified in Riemannian geometry, both in the global and infinitesimal settings \cite{KN,Ko}, see also \cite{Eg,I2}. A large number of other situations where
the gap phenomenon has been extensively studied falls in the framework of parabolic geometry 
\cite{CS}, see the results and historical discussion in \cite{KT}.

This article concerns symmetry in CR-geometry. While there was a considerable progress 
for Levi-nondegenerate CR-manifolds, in which case the geometry is parabolic, the problem of bounding  
symmetry dimension in general has been wide open. 

\subsection{The status of knowledge}

Recall that an {\it almost CR-structure}\, on a smooth manifold $M$ is a subbundle $H(M)\subset T(M)$ of the tangent bundle, called the CR-distribution, endowed with a field of operators 
$J_x:H_x(M)\rightarrow H_x(M)$, $J_x^2= -\hbox{id}$, smoothly depending on $x\in M$.
CR-dimension of  $M$ is $\hbox{CRdim}\,M=\frac12\rank H(M)$,
CR-codimension of $M$ is $\dim M-\rank H(M)$. The complexified CR-distribution splits
$H(M)\otimes\C=H^{(1,0)}(M)\oplus H^{(0,1)}(M)$, where
 $$
H_x^{(1,0)}(M)=\{X-iJ_xX\mid X\in H_x(M)\},\ \
H_x^{(0,1)}(M)=\{X+iJ_xX\mid X\in H_x(M)\}.
 $$
The almost CR-structure on $M$ is said to be integrable if the distribution $H^{(1,0)}(M)$ is involutive.
An integrable almost CR-structure is called a {\it CR-structure}\, and a manifold equipped with a CR-structure a {\it CR-manifold}. In this paper we consider only {\it CR-hypersurfaces}, i.e.\ 
CR-manifolds of CR-codimension 1.

A real hypersurface $M$ in a complex manifold $({\mathcal M},{\mathcal J})$ has an induced CR-structure: $H_x(M)=T_x(M)\cap {\mathcal J}_x\,T_x(M)$ and $J_x={\mathcal J}|_{H_x(M)}$ for $x\in M$. Conversely, every analytic CR-hypersurface is locally realizable as such real hypersurface of
CR-dimension $\dim_{\C}{\mathcal M}-1$. In smooth situation a realization is not always possible, but in this
article we restrict to real analytic CR-structures and hence make no distinction between abstract and embedded CR-hypersurfaces.

The {\it Levi form}\, of a CR-hypersurface $M$ at $x$ is given by the formula
${\mathcal L}_M(x)(\zeta,\overline{\zeta'})=i[{\mathfrak z},\overline{{\mathfrak z}'}](x)\,\hbox{mod}\, H_x(M)\otimes\C$, $\zeta,\zeta'\in H_x^{(1,0)}(M)$, where ${\mathfrak z}$, ${\mathfrak z}'$ are 
arbitrary local sections of $H^{(1,0)}(M)$ near $x$ such that 
${\mathfrak z}(x)=\zeta$, ${\mathfrak z}'(x)=\zeta'$. 
By identifying $T_x(M)/H_x(M)$ with $\R$, this is a Hermitian form on the CR-distribution
defined up to a real scalar multiple.

As shown in classical works \cite{C,CM}, \cite{Ta1,Ta2}, see also \cite{BS,CS},
the dimension of the symmetry algebra $\mathfrak{s}(M)$ of a {\it Levi-nondegenerate}\, connected CR-hypersurface $M$ of CR-dimension $n$ does not exceed $n^2+4n+3$. If $\dim\mathfrak{s}(M)$ attains this bound then $M$ is {\it spherical}, i.e.\ locally CR-equivalent to an open subset of the hyperquadric
 \begin{equation}\label{QK}
{\mathcal Q}_k=\Bigl\{(z,w)\in\C^n\times\C:  \Im w=\sum_{j=1}^k|z_j|^2-\sum_{j=k+1}^{n}|z_j|^2\Bigr\}
 \end{equation}
for some $0\le k\le n/2$. The Levi form of ${\mathcal Q}_k$ has signature $(k,n-k)$ everywhere and $\dim\mathfrak{s}({\mathcal Q}_k)=n^2+4n+3$ for all $k$. Thus, for the class of Levi-nondegenerate connected CR-hypersurfaces of CR-dimension $n$ one has $D_{\hbox{\tiny\rm max}}=n^2+4n+3$. Further, $D_{\hbox{\tiny\rm smax}}=n^2+3$ in the strongly pseudoconvex (Levi-definite) case for $n>1$ and $D_{\hbox{\tiny\rm smax}}=n^2+4$ in the Levi-indefinite case \cite{K2}. The situation $n=1$ is exceptional with $D_{\hbox{\tiny\rm smax}}=3$ \cite{C,KT}.

In the absence of Levi-nondegeneracy, finding the maximal and submaximal dimensions of the symmetry algebra is much harder. As is customary, assume the CR-manifold $M$ and the vector fields forming 
the symmetry algebra to be real-analytic. Then $\mathfrak{s}(M)=\hol(M)$ is finite-dimensional
provided that $M$ is {\it holomorphically nondegenerate}, see \cite[\S 11.3, \S12.5]{BER}, \cite{E,St}. Regarding the maximal symmetry dimension $D_{\hbox{\tiny\rm max}}$, the following is a variant of Beloshapka's conjecture, cf.~\cite[p.~38]{B2}. The authors of \cite{KS2} argument that for $n=1$ this is a version of
Poincar\'e's probl\`eme local \cite{Po}.

 \begin{conjecture}\label{beloshapka}
For any real-analytic connected holomorphically nondegenerate CR-hypersurface $M$ of CR-dimension $n$ one has $\dim\mathfrak{s}(M)\le n^2+4n+3$, with the maximal value $n^2+4n+3$ attained only if on a dense open set $M$ is spherical.
 \end{conjecture}

For $n=1$ the above conjecture holds true since a 3-dimensional holomorphically nondegenerate CR-hypersurface always has points of Levi-nonde\-generacy. For $n=2$ the conjecture was established in \cite{IZ} where the proof relied on a reduction of 5-dimensional uniformly Levi-degenerate 2-nonde\-generate CR-structures to absolute parallelisms (see \cite[\S 11.1]{BER} for the definition of $k$-nondegeneracy). Thus, for real-analytic connected holomorphically nondegenerate CR-hypersurfaces of CR-dimen\-sion $1\leq n\leq 2$ one has,  just as in the Levi-nondegenerate case, $D_{\hbox{\tiny\rm max}}=n^2+4n+3$.

It was shown in \cite{KS2} that for $n=1$ the condition $\dim\hol(M,x)>5$ for $x\in M$ implies that $M$ is spherical near $x$, where $\hol(M,x)$ is the Lie algebra of germs at $x$ of real-analytic vector fields on $M$ whose flows consist of CR-transformations. In \cite{IK1} we gave a short proof of this fact, and, applying the argument of \cite{IK1} to the symmetry algebra $\mathfrak{s}(M)$ instead of $\hol(M,x)$, one also obtains $D_{\hbox{\tiny\rm smax}}=5$. Notice that the result of \cite{KS2,IK1} improves on the statement of Conjecture \ref{beloshapka} for $n=1$ by replacing the assertion of generic sphericity of $M$ by that of sphericity everywhere.

Further, in the recent paper \cite{IK2} we considered the case $n=2$. It was shown that in this situation
either $\dim\mathfrak{s}(M)=15$ and $M$ is spherical, or $\dim\mathfrak{s}(M)\le11$ with the equality occurring
only if on a dense open subset $M$ is spherical with Levi form of signature $(1,1)$. This result
improves on the statement of Conjecture \ref{beloshapka} for $n=2$ as it yields sphericity near every point of $M$.
In addition, we constructed a series of examples of pairwise nonequivalent CR-hypersurfaces with
$\dim\mathfrak{s}(M)=11$ thus establishing $D_{\hbox{\tiny\rm smax}}=11$.
This fact also led to the following analogue of the result of \cite{KS2} for $n=2$: the condition $\dim\hol(M,x)>11$
for $x\in M$ implies that $M$ is spherical near $x$, and this estimate is sharp.

\subsection{Main results}

In the present paper we assume that $n$ is arbitrary and that the Levi-nondegeneracy locus is nonempty,
which is a condition stronger than holomorphic nondegeneracy. Of course, in this case $M$ is Levi-nondegenerate
on a dense open subset of $M$, perhaps with different Levi-signatures at different points, and the symmetry dimension
is finite. One of our goals is to determine the maximal and submaximal dimensions in this situation.

 \begin{theorem}\label{main1}
Assume that $M$ is a real-analytic connected CR-hypersurface of CR-dimension $n\ge 1$ having a point of Levi-nondegeneracy. Then for its symmetry algebra $\mathfrak{s}=\hol(M)$ exactly one of the two situations is possible:
 \begin{itemize}
\item[{\rm (i)}] $\dim\mathfrak{s}=n^2+4n+3$ and $M$ is spherical everywhere,
\item[{\rm (ii)}] $\dim\mathfrak{s}\le n^2+2n+2+\delta_{2,n}$ and in the case of equality $M$ is spherical
on its Levi-nondegeneracy locus with fixed signature of the Levi form.
 \end{itemize}
Moreover, the upper bound in (ii) is realizable and so the submaximal dimension is
$D_{\hbox{\tiny\rm smax}}=n^2+2n+2+\delta_{2,n}$.
 \end{theorem}

This result improves on the statement of Conjecture \ref{beloshapka}. Note that the result is global in $M$,
even if one takes $M=U$ to be a small fixed neighborhood of a point $x\in M$.
The proof of the theorem also leads to the following local version of the result, generalizing
theorems from \cite{KS2,IK1,IK2} for arbitrary $n$.

 \begin{corollary}\label{Cor1}
With the assumptions of Theorem \ref{main1} in the case $n\ge 3$ the condition $\dim\hol(M,x)>n^2+2n+2$ for
$x\in M$ implies that $M$ is spherical in a neighborhood of the point $x$, and this estimate is sharp.
 \end{corollary}

As in papers \cite{IK1,IK2}, our argument relies on the techniques from Lie theory,
notably on the description of proper subalgebras of maximal dimension of $\mathfrak{su}(p,q)$
obtained in Theorem \ref{subalgmaxdim}, where $1\le p\le q$, $p+q\ge 3$.
These pseudo-unitary algebras are precisely the maximal symmetry algebras of spherical models.
We show that among proper maximal subalgebras of those the maximal dimension is attained
on certain parabolic subalgebras. This raises the question if all parabolic subalgebras can be
symmetries of CR-hypersurfaces. To this we answer affirmatively as follows.


 \begin{theorem}\label{main2}
All maximal parabolic subalgebras of the pseudo-unitary algebra $\mathfrak{su}(p,q)$ are realizable
as the symmetry of a certain blow up of the standard hyperquadric \eqref{QK}.
 \end{theorem}

We suggest that other (non-maximal) parabolic subalgebras can be realized as symmetries of iterated blow-ups,
and we demonstrate this in the first non-trivial case of CR-dimension $n=2$. This contributes to the
models with large symmetry algebras considered in \cite{IK2}.
The general problem is discussed in the conclusion of the paper.

Note that a blow-up construction in CR-geometry has been discussed so far only phenomenologically \cite{KS1,KL},
and even a formal definition of this procedure was lacking in general (so rather a blow-down has been
identified in loc.cit.).
We approach the general problem in Section \ref{S21}. The relation of such blow-up to symmetry is not straightforward.
We discuss it in Sections \ref{S22}-\ref{S23}. For instance, we will show that an
iterative blow-up (which can be considered as one blow-up from  the na\"{\i}ve topological viewpoint)
can reduce the symmetry beyond expectations.

It is not true that all sub-maximally symmetric models can be obtained by the proposed blow-up construction.
This concerns the series of models in \cite{IK2} and we construct more examples in Section \ref{S41}.
Actually, Theorem \ref{Model0} gives a series of examples of pairwise nonequivalent CR-hypersurfaces with
the submaximal value $\dim\mathfrak{s}(M)=n^2+2n+2$ for $n\neq2$.
However all examples we constructed and investigated can be shown (in many cases a-posteriori) to be obtained
by a blow-up with an additional ramified covering that we describe in Section \ref{S43}.
This gives a new powerful tool for generating symmetric models in CR-geometry.

Finally, let us characterize Lie groups of automorphisms with large dimensions.

 \begin{theorem}\label{main3}
Under the assumptions of Theorem \ref{main1} the automorphism group $G=\op{Hol}(M)$ satisfies one of the  alternatives:
 \begin{itemize}
\item[{\rm (i)}] $\dim G=n^2+4n+3$ and $M$ is spherical everywhere,
\item[{\rm (ii)}] $\dim G\le n^2+2n+2+\delta_{2,n}$ and in the case of equality $M$ is spherical
on its Levi-nondegeneracy locus with fixed signature of the Levi form.
 \end{itemize}
The upper bound in (ii) is realizable, implying that the submaximal dimension of
the automorphism group is the same $D_{\hbox{\tiny\rm smax}}$ as in the Lie algebra case.
 \end{theorem}

The structure of the paper is as follows. In Section \ref{S2} we introduce the CR blow-up,
as our main tool to create examples, and we 
construct some models with large symmetry algebra/group of automorphims.
In Section \ref{S3}, using the algebraic and analytic techniques, we derive a sharp upper bound on the symmetry
dimension, thus proving the maximal and submaximal symmetry bounds; the reader interested in the gap phenomenon
can proceed directly there. Then in Section \ref{S4} we
provide further examples, containing an infinite sequence of submaximally symmetric and other
models with large symmetry. Finally, in the Conclusion we formulate a more general conjecture
on the symmetry dimension of CR-hypersurfaces and discuss other relevant problems.

\bigskip

\noindent\textsc{Acknowledgements.}
My first and foremost thanks go to Alexander Isaev, who influenced several results in this paper.
Our correspondence was of invaluable help. He brought to my attention a discussion,
where Stefan Nemirovski suggested a method for blowing up hyperquadrics in order to construct CR-manifolds
with submaximal symmetry dimension. I am grateful for this idea, inspiring the following progress.
The suggestion that blow-ups can be useful for submaximal symmetry models was also independently communicated
to the author by Ilya Kossovskiy.

The {\tt DifferentialGeomet\-ry} package of {\tt Maple} was used for extensive experiments
that 
 underly symmetry computations for all models in this paper.

\section{The blow-up construction}\label{S2}

Recall a construction from affine geometry. Let $L$ be a subspace of a vector space $V$, $\op{codim}(L,V)=m$.
The blow-up of $V$ along $L$ (below $\Pi$ is a subspace and $x$ a point) is
 $$
\op{Bl}_LV=\{(x,\Pi):x\in\Pi\supset L;\op{codim}(L,\Pi)=1\}.
 $$
This works over any field, in particular for complex $V,L,\Pi$ the blow-up is a 
complex algebraic manifold.
The projection $\pi_L:\op{Bl}_LV\to V$, $(x,\Pi)\mapsto x$, is a biholomorphism when restricted
to $\pi_L^{-1}(V\setminus L)$, and $\pi_L^{-1}(x)=\mathbb{P}(V/L)\simeq\C P^{m-1}$ for $x\in L$.

The construction canonically extends to complex analytic geometry: if $L$ is a complex submanifold of 
a complex manifold $V$, apply the above formula using local charts $V\supset U_\alpha\simeq\C^n$, 
straightening $L\cap U_\alpha$ and patching the charts to obtain $\op{Bl}_LV$, see e.g.~\cite{H}. 
The projection $\pi_L:\op{Bl}_LV\to V$ is holomorphic and satisfies the same properties:
$\pi_L^{-1}(V\setminus L)\simeq V\setminus L$ and $\pi_L^{-1}(x)\simeq\C P^{m-1}$ for $x\in L$.
Everywhere below we will assume that $m=\op{codim}L>1$, because otherwise $\op{Bl}_LV\simeq V$ for $m=1$.

Our aim is to extends this construction from complex geometry to CR-geometry.
Though such a construction can be given on the abstract level, it is convenient to
present a version for embedded CR-surfaces and we restrict to hypersurfaces.
In this section we formulate only the standard blow-up; variations on it, like iterated blow-ups,
weighted blow-ups and ramified coverings will be discussed in Section \ref{S43}.

\subsection{Blow-up in CR-geometry}\label{S21}

Let $\iota:M\hookrightarrow V$ be a real hypersurface in a complex manifold of dimension $n+1$
and $\pi_L:\op{Bl}_LV\to V$ a blow-up along a complex submanifold $L$ meeting $M\equiv\iota(M)$. 
In general, $L$ does not belong to $M$ and the germ of $L$ along $M\subset V$ is uniquely determined 
by $L'=L\cap M$. Define
 $$
\tilde{M}=\pi_L^{-1}(M)=(M\setminus L')\cup\pi_L^{-1}(L')\subset\op{Bl}_LV.
 $$
This subset has singular points $\Sigma_{\tilde{M}}\subset\pi_L^{-1}(L')$. For our purposes it is enough
to describe singularities in an affine chart: $V=\C^{n+1}$ and $L\subset V$ a subspace.

 \begin{lem}\label{LX}
A point $\tilde{x}=(x,\Pi)\in\tilde{M}$ belongs to $\Sigma_{\tilde{M}}$ if and only if $x=\pi_L(\tilde{x})\in L'$ 
and $\Pi\subset H(x)$, where $H(x)$ is the CR-plane of $M$ at the point $x$. 
 \end{lem}
 
 \begin{proof}
  \comm{
Let $M$ be the zero set of a real-analytic function $f:V\to\R$.
Then the pullback $\tilde{f}=\pi_L^*f$ defines $\tilde M=\{\tilde{x}\in\op{Bl}_LV:\tilde{f}(\tilde{x})=0\}$.
Conversely, if $\tilde M$ is the zero set of a real-analytic function $\tilde{f}$ on $\op{Bl}_LV$, 
then $\tilde{f}$ vanishes on $\pi_L^{-1}(L')$ and so it can be pushed forward to a continuous function $f$ 
on $V$. Using complexification and the Hartogs principle 
we can analytically extend $f$ from $V\setminus L'$ to $V$ (by zero); 
in other words, $L'$ is a removable singularity for $f$. 
Thus a defining function for $\tilde M$ is always the pullback $\tilde{f}=\pi_L^*f$ of the defining 
function $f:V\to\R$ for $M$, which can be taken non-singular: $d_xf\neq0$ for all $x\in M$.
 } 
Let $M$ be the zero set of a non-singular function $f:V\to\R$, i.e.\
$d_xf\neq0$ for all $x\in M$. A point $\tilde{x}$ is critical for $\tilde{f}=\pi_L^*f$ if $d_{\tilde{x}}\tilde{f}=0$. 
Since $H(x)\subset T_xM=\op{Ker}(d_xf)$, the map
$d_{\tilde{x}}\tilde{f}=d_xf\circ d_{\tilde{x}}\pi_L:T_{\tilde{x}}\op{Bl}_LV\to\R$
factorizes through $T_xV/H(x)\simeq\C$. 

Since $\pi_L$ is a diffeomorphism outside $L$, we can restrict to $x\in L'$. With such $x$ one readily 
verifies that the image of $d_{\tilde{x}}\pi_L:T_{\tilde{x}}\op{Bl}_LV\to T_xV$ at $\tilde{x}=(x,\Pi)$
coincides with $\Pi$, and so it belongs to the kernel of $d_xf$ if and only if $\Pi\subset H(x)$.

Thus $\tilde{x}$ is non-singular unless $T_xL=T_xL'\subset\Pi\subset H(x)$. A priori it could happen 
that $\tilde{M}$ possesses another defining function $\tilde{f}$ near such $\tilde{x}$ that is not a pullback $\pi_L^*f$, yet a closer analysis shows that the singularity at $\tilde{x}$ is conical and hence essential.
 \end{proof}

 \begin{corollary}\label{crr}
Let $m=\op{codim}(L,V)$ and $x\in L'=M\cap L$. Then 
the fiber over $x$ is $\pi_L^{-1}(x)\simeq\C^{m-1}=\C P^{m-1}\setminus \C P^{m-2}$ if $T_xL\subset H(x)$ 
and $\pi_L^{-1}(x)\simeq\C P^{m-1}$ else. \hfill$\Box$
 \end{corollary}

Removing singularities from $\tilde{M}$ we obtain what we call CR-blowup of $M$ along $L$:
 $$
\op{Bl}_LM= \tilde{M}\setminus\Sigma_{\tilde{M}}
 $$
In particular, for $L=o\in M$ we obtain the CR-blowup of $M$ at the point $o$. 

 \begin{prop}
For real-analytic CR-hypersurfaces $M$ the CR-blowup construction is well-defined, i.e.\ a change of
the embedding $\iota$ results in a CR-equivalence of $\op{Bl}_LM$. 
Moreover, $\op{Bl}_LM$ is connected if $M$ is connected.
 \end{prop}

 \begin{proof}
Note at first that the construction is defined because every real-analytic CR-surface admits a
closed real-analytic CR-embedding as a hypersurface to a complex manifold $V$ \cite{AF}.
Next, by Theorem 1.12 of loc.cit.\
such an embedding is unique up to a biholomorphism of (the germ of) a neighborhood of $\iota(M)\subset V$.
Since biholomorphisms naturally induce maps of blow-ups the first claim follows.

The second claim of the proposition follows from Corollary \ref{crr}.
 \end{proof}

 \begin{examp}\label{Ex1}\rm
Let us blow-up the hyperquadric ${\mathcal Q}=\{\op{Im}(w)=\|z\|^2\}\subset\C^n(z)\times\C(w)$
at the point $o=(0,0)$, where $\|z\|^2=\sum_{j=1}^n\sigma_j|z_j|^2$, $\sigma_j=\pm1$, $z=(z_1,\dots,z_n)$.
The blow-up contains the following open dense subset
 $$
\op{Bl}_o{\mathcal Q}\supset M=\{\op{Im}(w)=|w|^2\cdot\|z\|^2\}\stackrel{\pi_o}\longrightarrow{\mathcal Q}
 $$
with $\pi_o(z,w)=(w\cdot z,w)$. The model $M$ for $n=1$ appeared in \cite{KS1}.

The whole blow-up is obtained from
 $$
\Bigl\{\bigl((z_1,\dots,z_n,w),[\zeta_1:\dots:\zeta_n:\varpi]\bigr)\in\C^{n+1}\times\C P^n\, , \,
\op{Im}(w)=\|z\|^2,\ \frac{z_1}{\zeta_1}=\dots=\frac{z_n}{\zeta_n}=\frac{w}{\varpi}\Bigr\}
 $$
by removing singularities. In the chart $\varpi\neq0$ we get $U_0=M$ as above.
For $1\leq k\leq n$ in the chart $\zeta_k\neq0$ we get
 $$
U_k=\bigl\{\op{Im}(z_kw)=\sum_{j\neq k}\sigma_j|z_jz_k|^2+\sigma_k|z_k|^2\bigr\}.
 $$
The singularities of $U_k$ are $\Sigma_k=\{z_k=0,w=0\}$, so $U_k'=U_k\setminus\Sigma_k$ is the nonsingular part.
The projections $\pi_o^k:U'_k\to{\mathcal Q}$ and the gluing maps $\varphi_k:U'_k\setminus\{w=0\}\to U_0$ 
are given by the formulae:
 \begin{gather*}
\pi_o^k(z_1,\dots,z_n,w)=(z_1z_k,\dots,z_{k-1}z_k,z_k,z_kz_{k+1},\dots,z_kz_n,z_kw),\\
\varphi_k(z_1,\dots,z_n,w)=
\Bigl(\frac{z_1}w,\dots,\frac{z_{k-1}}w,\frac1w,\frac{z_{k+1}}w,\dots,\frac{z_n}w,z_kw\Bigr).
 \end{gather*}
Thus $\op{Bl}_o{\mathcal Q}$ 
is obtained from the union of $U_0,U_1,\dots,U_n$ by gluing via $\varphi_1,\dots,\varphi_n$.
Since $\cup_{k=1}^n\pi_o^k(U_k')\cap\{w=0\}$ (this is empty only for the sign definite norm $\|z\|^2$)
is the null-cone $\{\|z\|^2=0,z\neq0,w=0\}={\mathcal Q}\setminus\pi_o(M)$, we obtain 
$$\op{Bl}_o{\mathcal Q}={\mathcal Q}\cup M$$
In what follows we often change the blow-up $\op{Bl}_o{\mathcal Q}$ to $M$.
 \end{examp}

 \begin{examp}\label{Ex2}\rm
More generally, let $\C^n(z)=\C^{n-k}(z')\times\C^{k}(z'')$ be the direct product
and let $\|z\|^2=\|z'\|^2+\|z''\|^2$ be the quadric of signature $(\bar{p},\bar{q})$,
where both quadrics $\|z'\|^2$ and $\|z''\|^2$ are nondegenerate of signatures $(p',q')$ and $(p'',q'')$ with
$p'+p''=\bar{p}$, $q'+q''=\bar{q}$. Let $L=\C^{k}(z'')$. The corresponding blow-up contains the following model
 $$
\op{Bl}_L{\mathcal Q}\supset M=\{\op{Im}(w)=|w|^2\cdot\|z'\|^2+\|z''\|^2\}\stackrel{\pi_L}\longrightarrow{\mathcal Q}
 $$
with $\pi_L(z',z'',w)=(w\cdot z',z'',w)$.
 \end{examp}

\subsection{Symmetry of a Blow-up}\label{S22}

Next we describe how symmetry algebra of $M$ changes upon a blow-up construction.
Recall that the Levi-degeneracy locus in $M$ is an analytic subset.

 \begin{theorem}\label{symBlUp}
Let $M$ be a connected real analytic CR-hypersurface having Levi non\-degenerate points. 
If $L'\neq L$ assume that either each component of $L'$ contains a Levi-nondegenerate point
or that the Levi-degeneracy locus in $M$ has $\op{codim}>1$.
 
Then the symmetry algebra of the blow-up $\mathfrak{s}(\op{Bl}_LM)$ is the subalgebra in the Lie algebra 
$\mathfrak{s}(M)$ consisting of symmetries preserving $L'$, i.e.\ tangent to $L$ along $M$.
The same is true for the germs of symmetries, i.e.\ $\hol(\op{Bl}_LM,\tilde{x})\subset\hol(M,\pi_L(\tilde{x}))$
is determined by the condition to preserve $L'$ in a neighborhood of $\pi_L(\tilde{x})$.
 \end{theorem}

 \begin{proof}
By \cite[Theorem 1.12]{AF} and \cite[Proposition 12.4.22]{BER},
the infinitesimal symmetries of $M$ are bijective with holomorphic vector fields on the germ
of $M$ in $V$ that along $M$ are tangent to $H(M)$, the holomorphic tangent bundle of $M$.
In other words, every real-analytic infinitesimal CR-automorphism defined on an open subset $U'\subset M$
is the real part of a holomorphic vector field defined on an open subset $U\subset V$ with
$\iota(U')\subset\iota(M)\cap U$. The condition in the theorem is easily verified to be independent of 
the choice of CR-embedding (realization) $\iota:M\to V$.

Now we claim that $\op{Bl}_LM$ is Levi-degenerate along $\pi^{-1}_L(L')$.
Recall that the Levi form ${\mathcal L}_M$ of $M$ at $x$ can be identified with $i\p\bar\p f|_{H(x)}$, where 
$f$ is the defining function of $M$ in $V$. Similarly, the Levi form of the blow-up at $\tilde{x}$ is
$i\p\bar\p\tilde{f}|_{H(\tilde{x})}$, where $\tilde{f}=f\circ\pi_L$. Since $\pi_L$ is holomorphic we get
${\mathcal L}_{\op{Bl}_LM}(\tilde{x})=i\p\bar\p f\circ d\pi_L|_{H(\tilde{x})}$.
We already noted in the proof of Lemma \ref{LX} that the image of $d_{\tilde{x}}\pi_L$ at 
$\tilde{x}=(x,\Pi)$ with $x=\pi_L(\tilde{x})\in L'$ is $\Pi\subset T_xM$. 
Thus in the case when $m=\op{codim}(L)>2$ the rank of the Levi form is at most $\dim\Pi<2n$ 
and so ${\mathcal L}_{\op{Bl}_LM}(\tilde{x})$ is degenerate. For $m=2$ 
the same argument works if $T_xL\not\subset T_xM$ because then $\Pi\not\subset H(x)$ 
and the rank does not exceed $\dim(H(x)\cap\Pi)<2n$. 

In the case $T_xL\subset T_xM$ the points $\tilde{x}=(x,\Pi)$, where $\Pi\subset H(x)$, correspond to 
the singularity stratum $\Sigma_{\tilde{M}}$ that is removed, and the argument applies as well. Alternatively, 
we note that when $L\subset M$ the blow-up contains the complex hypersurface 
$\pi_L^{-1}(L)\simeq \C^{m-1}\times L$ and so $\op{Bl}_LM$ is not minimal along it.
Our further arguments can be applied component-wise, so we can assume $L'$ (and $L$) to be connected.

If $L'=L$ then $\dim\pi^{-1}_L(L)=2n$. If a symmetry is not tangent to this complex submanifold, then
a flow along it generates an open subset of Levi-degenerate points, which is impossible because
Levi-nondegenerate points are dense in $M$. 
If $L'\subsetneq L$ then $\dim\pi^{-1}_L(L')=2n-1$. Any symmetry must be therefore tangent to 
$\pi^{-1}_L(L')$ if Levi degeneracy locus has $\op{codim}>1$. Alternatively, if
$L'$ contains a Levi-nondegenerate point $x$, then any point from a small neighborhood $U_x\subset M$ 
of $x$ is Levi-nondegenerate. Thus a symmetry must be tangent to $\pi_L^{-1}(L'\cap U_x)$
and hence, by analyticity of vector fields from $\mathfrak{s}(M)$, this symmetry is tangent 
to $\pi_L^{-1}(L')$ everywhere.

We conclude that in any case the symmetries of $\op{Bl}_LM$ must be tangent to $\pi^{-1}_L(L')$.
Thus they descend to the blow-down manifold $M$. 
Indeed, consider a symmetry $s\in\mathfrak{s}(\op{Bl}_LM)$
restricted to $\op{Bl}_LM\setminus\pi^{-1}_L(L')\simeq M\setminus L'$. Choose a neighborhood 
$U\subset V$ of $x\in L'$ and adapted complex coordinates to the submanifold $L\subset V$. 
In these coordinates the components of $s$ are holomorphic functions that analytically extend to $L$
by the Hartogs principle applied to $U$ (it is important here that $m=\op{codim}L>1$). 
In other words, $L'$ is a removable singularity for the symmetry $s$ on $M$. 

Therefore we get a map $q_L:\mathfrak{s}(\op{Bl}_L(M))\to\mathfrak{s}(M)$, which
is clearly a homomorphism of Lie algebras.
Since $\pi_L:\op{Bl}_LM\to M$ is a biholomorphism over $M\setminus L'$ the map $q_L$ is injective.
Indeed, $q_L(\tilde{s})=0$ for $\tilde{s}\in\mathfrak{s}(\op{Bl}_LM)$ implies $\tilde{s}|_U=0$ for
$U\subset M\setminus L'$ and therefore $\tilde{s}=0$ by analyticity and connectedness of $\op{Bl}_LM$.

It is clear that the vector fields $s\in\mathfrak{s}(M)$ that lift to $\op{Bl}_LM$ must 
preserve $L'$. Conversely, if $s$ preserves $L'$ it lifts to the blow up of $V$ along $L$. Since $s$ is also 
a symmetry of $M$, it restrict to $\tilde{M}$ and then to the non-singular part $\op{Bl}_LM$. Thus $q_L$ 
has the required image: $q_L(\mathfrak{s}(\op{Bl}_L(M)))=\{s\in\mathfrak{s}(M):s(x)\in T_xL'\,\forall x\in L'\}$.

The proof in the case of germs of symmetries is completely analogous.
 \end{proof}
 
 \begin{remark}\label{referee} \rm
The Levi-nondegeneracy assumptions in Theorem \ref{symBlUp} can be relaxed to 
$k$-nondegeneracy for $k>1$, as was kindly communicated to us by a reviewer:
A combination of \cite[Theorem 1.1,Theorem 1.4]{ES} implies that every point
$\tilde{x}$ of $\pi^{-1}_L(L')$ is not of finite type in the sense of Kohn and Bloom-Graham, so 
not finitely-nondegenerate by \cite[Remark 11.5.14]{BER}. 
On holomorphically nondegenerate $M$ every point in the complement to a proper analytic set
is finitely-nondegenerate \cite[Theorem 11.5.1]{BER}.
 \end{remark}

 \begin{examp}\label{Ex3}\rm
The symmetry algebra of the hyperquadric ${\mathcal Q}$ is $\mathfrak{su}(p,q)$,
where $(\bar{p},\bar{q})=(p-1,q-1)$ is the signature of the Levi form. The isotropy algebra of a point is
the first parabolic subalgebra $\mathfrak{p}_{1,n+1}$, and hence this is the symmetry
of the blow-up model $\op{Bl}_o{\mathcal Q}$ constructed in Example \ref{Ex1}.
In Section \ref{S42} we will give explicit formulae for the symmetry fields of the open dense submanifold
$M\subset \op{Bl}_o{\mathcal Q}$.
 \end{examp}

This example implies the following statement.

 \begin{corollary}\label{realizationP}
The first parabolic subalgebra $\mathfrak{p}_{1,n+1}\subset\mathfrak{su}(p,q)$ is realizable as symmetry
of an analytic CR-hypersurface of\, $\op{CRdim}=n$ containing Levi nondegenerate points.
As such one can take either the constructed blow-up or its submanifold $M\subset \op{Bl}_o{\mathcal Q}$.
 \end{corollary}

 \begin{proof}
That the symmetry of $\op{Bl}_o{\mathcal Q}$ is as indicated follows from Theorem \ref{symBlUp}.
Let us also show that the symmetry does not grow upon restriction to the submanifold $M$.
The subset $\pi_o(M)\subset{\mathcal Q}$ is obtained from the quadric by removing the hyperplane $\{w=0\}$
punctured at $o$. The symmetry algebra of both $\pi_o(M)$ and ${\mathcal Q}$ is $\mathfrak{su}(p,q)$.
Now the same argument as in the above proof shows that $q_L:\mathfrak{s}(M)\to\mathfrak{s}(\pi_o(M))$
is an injective map with the image consisting of symmetry fields vanishing at $o$.
 \end{proof}

 \begin{examp}\label{Ex4}\rm
Considering the more general blow-up model from Example \ref{Ex2} with $k=\dim_\C L\in(0,n)$
we conclude that its symmetry is smaller in dimension than the parabolic subalgebra fixing a subspace of dimension
$(k+1)$ in linear representation.
For instance, for $\bar{p}=\bar{q}=1$ the symmetry algebra $\mathfrak{s}(\op{Bl}_L{\mathcal Q})$
has dimension 8, while the corresponding parabolic algebra $\mathfrak{p}_2\subset\mathfrak{su}(2,2)$
has dimension 11. This is in accordance with Theorem \ref{symBlUp},
if one verifies the action of $\mathfrak{s}({\mathcal Q})$ on $L$.
 \end{examp}

\subsection{Automorphisms of a Blow-up}\label{S23}

The argument of the previous theorem extends to the Lie group case and we get:

 \begin{theorem}\label{symBlUp2}
Let $M$ be a real analytic CR-hypersurface having Levi-nondegenerate points.
Let $G$ be the CR-automorphism group of $M$, and $\tilde{G}$ be the CR-automorphism group 
of $\op{Bl}_L(M)$. Let subgroups $G_0$, $\tilde{G}_0$ be their components of unity.

If $L'$ satisfies the assumption of Theorem \ref{symBlUp}, then
$\tilde{G}_0$ is the stabilizer of (each component of) $L'$ in $G_0$.
If, in addition, $M$ is minimal in the case $L'=L$ or, alternatively,
the Levi degeneracy locus of $M$ has $\op{codim}>2$, then
$\tilde{G}$ is the stabilizer of $L'$ in $G$. \hfill$\Box$
 \end{theorem}
The second statement follows by dimension comparisson of Levi degeneracy loci.

\smallskip

We give an application of this theorem. Let $p+q=n+2$, $1\leq s\leq p\leq q$.
Recall that the parabolic subgroup $P_{s,n-s+2}\subset SU(p,q)$ is the stabilizer of a null $s$-plane
(and thus also of the orthogonal co-isotropic $(n-s+2)$-plane) in the standard representation of $SU(p,q)$
on $\C^{n+2}$, its Lie algebra is the parabolic subalgebra $\mathfrak{p}_{s,n-s+2}$.

 \begin{examp}\label{Ex5}\rm
Let $k=s-1$. Consider the hyperquadric ${\mathcal Q}_{p-1}\subset\C^n(z)\times\C(w)$ defined as
 $$
\op{Im}(w)=\sum_{j=1}^{k}\left(z_j\bar z_{j+k}+z_{j+k}\bar z_j\right)+\|z'\|^2,
 $$
where
 $$
z'=(z_{2k+1},\dots,z_n),
\qquad
\|z'\|^2=\sum_{\ell=2k+1}^{p-1+k}|z_{\ell}|^2-\sum_{\ell=p+k}^{n}|z_{\ell}|^2.
 $$
Let
 \begin{equation}\label{L-mod1}
L=\{(z,w)\in\C^{n+1}:z_j=0\ (1\leq j\leq k),\ z_{\ell}=0\ (2k+1\leq\ell\leq n),\ w=0\}.
 \end{equation}
Clearly, $L$ has dimension $k$ and lies in ${\mathcal Q}_{p-1}$. An open dense subset $M$ of
the blow-up $\op{Bl}_L{\mathcal Q}_{p-1}$ belongs to the hypersurface $S\subset\C^n(z)\times\C(w)$ given by
 \begin{equation}\label{Eq-mod1}
\op{Im}(w)=\sum_{j=1}^{k}\left(z_jw\bar z_{j+k}+z_{j+k}\bar w\bar z_j\right)+|w|^2\cdot\|z'\|^2
 \end{equation}
with the projection $\pi_L:S\to{\mathcal Q}_{p-1}$ given by
 \begin{equation}\label{pi-mod1}
\pi_L(z_1,\dots,z_n,w)=(z_1w,\dots,z_kw, z_{k+1},\dots,z_{2k},z_{2k+1}w,\dots,z_nw,w).
 \end{equation}
The hypersurface $S$ contains the hyperplane $\{w=0\}=\pi_L^{-1}(L)$, and for every $x\in L$ the fiber
$\pi_L^{-1}(x)$ is an $(n-k)$-dimensional vector subspace of $\C^{n+1}$. The singular locus of $S$ is given
by
 $$
\Sigma_S=\Bigl\{\sum_{j=1}^{k}z_j\bar z_{j+k}=-\frac{i}{2},\ w=0\Bigr\}.
 $$
The CR-hypersurface $M$ is obtained by excluding $\Sigma_S$ from $S$.
Thus $M\subset S$ is an open subset containing an open subset of the hyperplane $\{w=0\}$.

By Theorem \ref{symBlUp} the symmetry algebra $\mathfrak{s}(M)$ is $\mathfrak{p}_{s,n-s+2}$.
We do not provide details of this derivation here because in Section \ref{S42} we present these symmetries
explicitly. This will realize all maximal parabolic subalgebras of $\mathfrak{su}(p,q)$.
 \end{examp}

Note that the automorphism group of the spherical surface ${\mathcal Q}_{p-1}$ is not $SU(p,q)$,
but its parabolic subgroup $P_{1,n+1}$ due to incompleteness, and so the automorphism group of its
blow-up is not $P_{s,n-s+2}$ (even for $s=1$).

 \begin{examp}\label{Ex6}\rm
Let us consider the compact version, possessing the automorphism group of maximal size.
For this embed the hyperquadric ${\mathcal Q}={\mathcal Q}_{p-1}$ into projective space and take the closure:
 $$
\overline{{\mathcal Q}}=\left\{[z:w:\xi]\in\C\PP^{n+1}: \frac{w\bar\xi-\xi\bar w}{2i}=
\sum_{j=1}^{k}\left(z_j\bar z_{j+k}+z_{j+k}\bar z_j\right)
+\!\sum_{\ell=2k+1}^{p-1+k}|z_{\ell}|^2-\!\sum_{\ell=p+k}^{n}\!|z_{\ell}|^2\right\},
 $$
where the hyperplane at infinity is $\C\PP^n=\{\xi=0\}$.

The Lie group $G=\PSU(p,q)$ acts transitively on $\overline{{\mathcal Q}}$.
Moreover, it acts transitively on the manifold $N$ of linear subspaces of $\C\PP^{n+1}$ of dimension $k$
that lie in $\overline{{\mathcal Q}}$ with $\dim N=(k+1)(2n-3k+1)$.
The stabilizer of a point $L$ in $N$ is the parabolic subgroup $P_{s,n-s+2}\subset PSU(p,q)$,
and one can verify using Theorem \ref{symBlUp2} that this is indeed the automorphism group of
$\op{Bl}_L\overline{{\mathcal Q}}$.
 \end{examp}

\section{The gap phenomenon}\label{S3}

In this section we prove Theorem \ref{main1}, Corollary \ref{Cor1}, Theorem \ref{main3} and further results.

\subsection{An algebraic dimension bound}\label{S31}

Consider the simple Lie algebra $\mathfrak{su}(p,q)$, $1\le p\le q$, $p+q=n+2\ge 3$, where $p$ counts
the number of positive eigenvalues and $q$ the number of negative ones in the signature of the corresponding Hermitian form. The case of sign-definite metric, i.e.\ the algebra $\mathfrak{su}(n+2)$, will be excluded from consideration.

The algebra has type $A_{n+1}$, and  its parabolic subalgebra corresponding to the crossed nodes that form a subset $I$ of the nodes of the Satake diagram
is denoted by $\mathfrak{p}_I$. In particular, the maximal parabolic subalgebras are
$\mathfrak{p}_{s,n-s+2}$ for $1\leq s\leq p$, where for $n=2m-2$
we identify $\mathfrak{p}_{m,m}$ with $\mathfrak{p}_{m}$.
Recall that a cross can be imposed only on a white node of the Satake diagram; any two white nodes related by an arrow shall be crossed simultaneously. Here are some examples:
 \begin{align*}
&\mathfrak{p}_{1,2}\subset\mathfrak{su}(1,2):
{
 \begin{tiny}
 \begin{tikzpicture}[scale=0.8,baseline=-3pt]
\bond{0,0}; \DDnode{x}{0,0}{}; \DDnode{x}{1,0}{};
\node (A) at (0,0.1) {}; \node (B) at (1,0.1) {}; \path[<->,font=\scriptsize,>=angle 90] (A) edge [bend left] (B);
 \useasboundingbox (-.4,-.2) rectangle (2.2,0.55); 
 \end{tikzpicture}
 \end{tiny}
 }
&\mathfrak{p}_{1,3}\subset\mathfrak{su}(2,2):
{
 \begin{tiny}
 \begin{tikzpicture}[scale=0.8,baseline=-3pt]
\bond{0,0}; \bond{1,0}; \DDnode{x}{0,0}{}; \DDnode{w}{1,0}{}; \DDnode{x}{2,0}{};
\node (A) at (0,0.1) {}; \node (B) at (2,0.1) {}; \path[<->,font=\scriptsize,>=angle 90] (A) edge [bend left] (B);
 \useasboundingbox (-.4,-.2) rectangle (2.2,0.55); 
 \end{tikzpicture}
 \end{tiny}
 }
 \\
&\mathfrak{p}_{1,3}\subset\mathfrak{su}(1,3):
{
 \begin{tiny}
 \begin{tikzpicture}[scale=0.8,baseline=-3pt]
\bond{0,0}; \bond{1,0}; \DDnode{x}{0,0}{}; \DDnode{b}{1,0}{}; \DDnode{x}{2,0}{};
\node (A) at (0,0.1) {}; \node (B) at (2,0.1) {}; \path[<->,font=\scriptsize,>=angle 90] (A) edge [bend left] (B);
 \useasboundingbox (-.4,-.2) rectangle (2.2,0.55); 
 \end{tikzpicture}
 \end{tiny}
 }
&\mathfrak{p}_{2}\subset\mathfrak{su}(2,2):
{
 \begin{tiny}
 \begin{tikzpicture}[scale=0.8,baseline=-3pt]
\bond{0,0}; \bond{1,0}; \DDnode{w}{0,0}{}; \DDnode{x}{1,0}{}; \DDnode{w}{2,0}{};
\node (A) at (0,0.1) {}; \node (B) at (2,0.1) {}; \path[<->,font=\scriptsize,>=angle 90] (A) edge [bend left] (B);
 \useasboundingbox (-.4,-.2) rectangle (2.2,0.55); 
 \end{tikzpicture}
 \end{tiny}
 }
 \end{align*}

 \begin{proposition}\label{dimmaxparsubalg}
Dimension of the maximal parabolic subalgebra $\mathfrak{p}_{s,n-s+2}\subset\mathfrak{su}(p,q)$
is $d_n(s)=n^2-2sn+3s^2+4n-4s+3$.
 \end{proposition}

\begin{proof} For $\g=\mathfrak{su}(p,q)$ the grading $\g=\g_{-\nu}\oplus\dots\oplus\g_0\oplus\dots\oplus\g_{\nu}$
corresponding to a parabolic subalgebra $\mathfrak{p}\simeq\g_0\oplus\dots\oplus\g_{\nu}$ of $\g$
has $\g_0\simeq\mathfrak{z}(\g_0)\oplus\g_0^{ss}$, where the first summand is the center of dimension equal to the
number of crosses in the Satake diagram of $\g$ and the second (semisimple) summand corresponds to the Satake diagram obtained by removing the crosses.
Thus, a maximal parabolic subalgebra of $\g$, independently of the coloring of the nodes, satisfies:
 \begin{align*}
\hskip35pt
\dim\mathfrak{p}_{s,n-s+2}
&=\tfrac12(\dim\g+\dim\g_0)\\
&=\tfrac12(\dim A_{n+1}+2+2\dim A_{s-1}+\dim A_{n-2s+1})\\
&=\tfrac12((n+2)^2+(n-2s+2)^2)+s^2-1=d_n(s).
 \end{align*}
The case $n=2m$, $s=m+1$ is special yet subject to this formula.
 \end{proof}

Some initial values of $d_n(s)$ are as follows:
 $$\begin{tabular}{r|cccc}
\!\!${}_{n}\!\!\diagdown\!\!{}^s$\!\! & 1 & 2 & 3 & 4\\
\hline
1 & 5\\
\hline
2 & 10 & 11\\
\hline
3 & 17 & 16\\
\hline
4 & 26 & 23 & 26\\
\hline
5 & 37 & 32 & 33\\
\hline
6 & 50 & 43 & 42 & 47\\
\hline
7 & 65 & 56 & 53 & 56
 \end{tabular}$$

 \begin{corollary}\label{maxdimparab}
The maximal dimension of a parabolic subalgebra of $\mathfrak{su}(p,q)$ is uniquely given by
$\dim\mathfrak{p}_{1,n+1}=n^2+2n+2$ except for $n=2$ where the maximum is attained by $\dim\mathfrak{p}_2=11>\dim\mathfrak{p}_{1,3}$ and $n=4$ where
$\dim\mathfrak{p}_{3}=\dim\mathfrak{p}_{1,5}=26$.
 \end{corollary}

Now we restrict the dimension of a proper subalgebra of the pseudounitary algebra,
which simultaneously gives a bound for subgroups of the pseudounitary group.

 \begin{theorem}\label{subalgmaxdim}
A proper subalgebra of $\mathfrak{su}(p,q)$ of maximal dimension
is a parabolic subalgebra, as described in Corollary {\rm\ref{maxdimparab}}.
 \end{theorem}

 \begin{proof}
By Mostow's theorem \cite{M}, a maximal subalgebra of a real simple Lie algebra is either parabolic, or
the centralizer of a pseudotorus, or semisimple.

The centralizers of pseudotoric subalgebras of $\mathfrak{su}(p,q)$ have the maximal possible dimension for either
$\mathfrak{u}(p,q-1)$ or $\mathfrak{u}(p-1,q)$, both of dimension $(n+1)^2<\dim\mathfrak{p}_{1,n+1}$.

Next, fix a semisimple subalgebra $\h\subset \mathfrak{su}(p,q)$; by complexifying it we obtain a subalgebra $\h^\C\subset\mathfrak{su}(p,q)^\C=\mathfrak{sl}(n+2,\C)$.
By Dynkin's theorem (see \cite{D} and also \cite[Chap. 6, Sect. 3.2]{GOV}) a maximal semisimple subalgebra of the simple Lie algebra of type $A_{n+1}$
is either (i) nonsimple irreducible, or (ii) simple irreducible.

If $\h^{\C}$ falls in Case (i), we have $n+2=st$ ($1<s\leq t<n+2$, hence $n\ge 2$) and $\h^\C=\mathfrak{sl}(s,\C)\oplus\mathfrak{sl}(t,\C)$ is embedded in $\mathfrak{sl}(n+2,\C)$ via the representation on $\C^s\otimes\C^t=\C^{n+2}$. Then $\dim_\C\h^\C=s^2+t^2-2=s^2+\frac{(n+2)^2}{s^2}-2$. The maximum of the function $s^2+\frac{(n+2)^2}{s^2}-2$ on the interval $2\le s\le n+1$ is attained at $s=n+1$ and is clearly seen to be strictly less than $\dim\mathfrak{p}_{1,n+1}$.

In Case (ii) we first assume that $\h^\C$ is a classical Lie subalgebra of $\mathfrak{sl}(n+2,\C)$. If $\h^{\C}$
has type $A$, then $\dim_{\C}\h^{\C}$ is maximal if $\h^{\C}=\mathfrak{sl}(k+1,\C)\subset\mathfrak{sl}(n+2,\C)$,
$k\leq n$, which
does not give the optimal dimension as $\dim_\C\mathfrak{sl}(k+1,\C)\leq n^2+2n<\dim\mathfrak{p}_{1,n+1}$.

If $\h^{\C}$ has type $B$ or $D$, then $\dim_{\C}\h^{\C}$ is maximal if $\h^{\C}=\mathfrak{so}(n+2,\C)\subset\mathfrak{sl}(n+2,\C)$,
which does not give the optimal dimension since $\dim_\C\mathfrak{so}(n+2,\C)=\frac12(n^2+3n+2)<\dim\mathfrak{p}_{1,n+1}$.

Suppose that $\h^{\C}$ has type $C$ and write $n=2k+r$, where $r$ is either 0 or 1. Then $\dim_{\C}\h^{\C}$ is maximal if $\h^{\C}=\mathfrak{sp}(2k+2,\C)\subset\mathfrak{sl}(n+2,\C)$, which again does not give the optimal dimension as $\dim_{\C}\mathfrak{sp}(2k+2,\C)=(k+1)(2k+3)$. Indeed, this number is strictly less than $\dim\mathfrak{p}_{1,n+1}$ for $n\ne 2$ and is strictly less than $11=\dim\mathfrak{p}_{2}$ for $n=2$.

Consider now the exceptional Lie algebras. The representation $V$ of minimal dimension of $\mathfrak{g}_2=\op{Lie}(G_2)$
has dimension 7 ($V=R_{\lambda_1}$), so if $\h^{\C}=\mathfrak{g}_2$ we have $n\ge 5$. Hence $\mathfrak{g}_2$ does not give the optimal dimension
since $\dim_\C\mathfrak{g}_2=14<5^2+2\cdot5+2$.

Similarly, the representation $V$ of minimal dimension of the exceptional Lie algebra $\mathfrak{f}_4=\op{Lie}(F_4)$
has dimension 26 ($V=R_{\lambda_4}$), so if $\h^{\C}=\mathfrak{f}_4$ we have $n\ge 24$. Hence $\mathfrak{f}_4$ does not give the optimal dimension
since $\dim_\C\mathfrak{f}_4=52<24^2+2\cdot24+2$.

In the same way, we argue for the E-series: the representation $V$ of minimal dimension for $\mathfrak{e}_6$,
$\mathfrak{e}_7$, $\mathfrak{e}_8$ has dimension 27, 56, 248, respectively (and for $V$ we have, respectively,
$R_{\lambda_1}\simeq R_{\lambda_6}$, $R_{\lambda_7}$, $R_{\lambda_8}$ in Bourbaki's enumeration).
Hence, none of these algebras gives the optimal dimension since
$\dim_{\C}\mathfrak{e}_6=78<25^2+2\cdot25+2$, $\dim_{\C}\mathfrak{e}_7=133<54^2+2\cdot54+2$, $\dim_{\C}\mathfrak{e}_8=248<246^2+2\cdot246+2$.

Thus, all semisimple subalgebras of $\mathfrak{su}(p,q)$ have dimensions strictly smal\-ler than the maximal possible dimension of a parabolic subalgebra.
\end{proof}

\begin{remark}\label{remsubsub} \rm
By \cite[Proposition 2.1; Remark 2.5]{IK2}, for $n=2$ every proper subalgebra of $\mathfrak{su}(p,q)$ of
dimension $10=n^2+2n+2$ is also parabolic and conjugate to $\mathfrak{p}_{1,3}$.
\end{remark}

\subsection{Establishing the submaximal symmetry dimension}\label{S32} 

We assumed that $M$ has a point of Levi-nondegeneracy, which implies that $M$ is {\it holomorphically nondegenerate},
see \cite[Theorem 11.5.1]{BER}. The condition of holomorphic nondegeneracy for a real-analytic hypersurface
in complex space was introduced in \cite{St} and requires that for every point of the hypersurface there exists
no nontrivial holomorphic vector field tangent to the hypersurface near the point.
Extensive discussions of this condition can be found in \cite[\S 11.3]{BER}, \cite{E},
but we only make a note of the fact, stated in \cite[Corollary 12.5.5]{BER}, that the holomorphic nondegeneracy of $M$
is equivalent to the finite-dimensionality of all the algebras $\hol(M,x)$.
Notice that together with \cite[Proposition 12.5.1]{BER} this corollary implies that the finite-dimensionality
of $\hol(M,x_0)$ for some $x_0\in M$ implies  the finite-dimensionality of $\hol(M,x)$ for all $x\in M$.

Clearly, $\mathfrak{s}(M)=\hol(M)$ may be viewed as a subalgebra of $\hol(M,x)$ for any $x$.
Therefore for a holomorphically nondegenerate $M$, and in particular for the case we consider,
the symmetry algebra $\mathfrak{s}(M)$ is finite-dimensional.

\medskip

 \begin{proof}[Proof of Theorem \ref{main1}]
For $n=1$ the theorem was obtained in \cite{KS2,IK1}, for $n=2$ its stronger variant was proven in \cite{IK2},
so we assume that $n\ge 3$.

Let $S_M$ be the Levi-degeneracy locus of $M$. It is a proper real-analytic subset of $M$.
Then $U=M\setminus S_M$ is an open dense subset of $M$.

Choose a point $x\in U$. The natural map $\hol(M)\to\hol(U)\to\hol(M,x)$ is injective.
If $x$ is not spherical, \cite{K2} implies $\dim\hol(M)\le n^2+4$ that is less than $n^2+2n+2$.

Thus every point of $U$ is spherical. Then $\hol(M)$ is a subalgebra of $\mathfrak{su}(p,q)$ for some
$1\le p\le q$, $p+q=n+2$, and by Theorem \ref{subalgmaxdim} there is an alternative:
 \begin{itemize}
\item[(i)] $\hol(M)=\mathfrak{su}(p,q)$;
\item[(ii)] $\dim\hol(M)\le n^2+2n+2$.
 \end{itemize}

Consider Case (i) first. We will show that $M$ is spherical everywhere.
Since we already established sphericity on $U$, fix a point $x_0\in S_M$.
Consider the isotropy subalgebra $\hol_0(M)$ of $\hol(M)$ at $x_0$.
Clearly, $\dim \hol_0(M)\ge (n+2)^2-1-(2n+1)=n^2+2n+2$.
Hence, appealing to Theorem \ref{subalgmaxdim} once again, we see that one of the following holds:

 \begin{itemize}
\item[(ia)] $\dim \hol_0(M)=n^2+2n+2$;
\item[(ib)] $\hol_0(M)=\hol(M)=\mathfrak{su}(p,q)$.
 \end{itemize}

In Case (ia), the orbit of $x_0$ under the corresponding local action of the group $\SU(p,q)$ is open,
so it contains a spherical point $x\in U$, and hence $M$ is spherical near the point $x_0$ as well.

In Case (ib), by the Guillemin-Sternberg theorem \cite[pp.~113--115]{GS}, the action of the simple Lie algebra
$\mathfrak{su}(p,q)$ is linearizable near $x_0$, and we obtain a nontrivial $(2n+1)$-dimensional representation
of $\mathfrak{su}(p,q)$. But the lowest-dimensional representation of $\mathfrak{su}(p,q)$ is the standard
$\C^{p,q}$ of real dimension $2n+4$, which is a contradiction.

Consider now Case (ii) and assume that $\dim\hol(M)=n^2+2n+2$. Then by Theorem \ref{subalgmaxdim} the algebra $\hol(M)$ is isomorphic either to the parabolic subalgebra $\mathfrak{p}_{1,n+1}$ of $\mathfrak{su}(p,q)$, or, if $n=4$ and $p=q=3$, to the parabolic subalgebra $\mathfrak{p}_3$ of $\mathfrak{su}(3,3)$. As all such parabolic subalgebras are pairwise nonisomorphic, we see that $p$ and $q$ are determined uniquely. Therefore, the Levi form of $M$ has fixed signature on $U$.

Finally, the obtained upper bound for the symmetry dimension is realizable due to Corollary \ref{realizationP}.
This finished the proof.
 \end{proof}

\medskip

 \begin{proof}[Proof of Corollary \ref{Cor1}]
If $M$ is holomorphically nondegenerate, then for every $x\in M$ there exists a connected neighborhood $U$
of $x$ in $M$ for which the natural map $\hol(U)\to\hol(M,x)$ is surjective \cite[Proposition 12.5.1]{BER};
for any such $U$ we have $\hol(M,x)=\hol(U,x)=\hol(U)$. Taking $U$ instead of $M$ in Theorem \ref{main1},
the statement of the corollary follows.
 \end{proof}

\subsection{Some results on spherical points}\label{S33}

Let us further discuss the result of Theorem \ref{main1}.
First note that the exceptional case $n=2$ can be included into part of the statement as follows.

 \begin{prop}\label{main1+}
Assume that $M$ is a real-analytic connected CR-hypersurface of CR-dimension $n\ge 1$ having a
Levi-nondegenerate point. If $\dim\mathfrak{s}(M)\ge n^2+2n+2$, then $M$ is spherical on its
Levi-nondegeneracy locus with fixed signature of the Levi form.
 \end{prop}

 \begin{proof}
Only the case $n=2$ is special in regard to the proof of Section \ref{S32}. In this case
$\dim\mathfrak{s}=10=n^2+2n+2$, and the statement follows by Remark \ref{remsubsub}
since the parabolic subalgebras $\mathfrak{p}_{1,3}\subset\mathfrak{su}(1,3)$ and
$\mathfrak{p}_{1,3}\subset\mathfrak{su}(2,2)$ derived in \cite{IK2} are not isomorphic.
 \end{proof}

Next, set
\vskip-20pt

 $$
d_0=\left\{\begin{array}{cl}
3 & \hbox{if $n=1$},\\ \vspace{-0.3cm}\\ n^2+4 & \hbox{if $n>1$.}
\end{array}\right.
 $$

 \begin{prop}\label{main1++}
Under the assumption of Proposition \ref{main1+} the inequality $\dim\mathfrak{s}>d_0$ implies that $M$
is spherical on its Levi-nondegeneracy locus, possibly with different signatures of the Levi form at different points.
 \end{prop}

 \begin{proof}
Let $S$ be the Levi-nondegeneracy locus of $M$. If there exists a point of $S$ near which $M$ is not spherical, then, since the natural map $\hol(M)\to\hol(M,x)$ is injective for every $x\in M$, by \cite{C}, \cite{K2} we have $\dim\hol(M)\le d_0$.
 \end{proof}

 \begin{remark}\rm
Concerning Proposition \ref{main1++}, \cite[Example 6.2]{KS1} actually shows that it is possible for the
Levi-nondegeneracy locus $S$ of a real-analytic CR-hypersurface $M$ to be disconnected,
for the signature of the Levi form of $M$ to be different on different connected components of $S$,
and for $M$ to be locally CR-equivalent to different hyperquadrics near different points.
By Proposition \ref{main1+} such an effect is impossible if the algebra $\hol(M)$ has large dimension.

The hypersurface $M\subset\C^3$ from \cite[Example 6.2]{KS1} is given by the equation
$$
\bar{w}=w\left(\frac{i|z_1|^2-\sqrt{1+2i|z_2|^2w-|z_1|^4}}{1+2i|z_2|^2w}\right)^2.
 $$
We found that its symmetry algebra is spanned by the vector fields
 \begin{gather*}
R=2\op{Re}(-z_2\p_{z_2}+2w\p_w),\ S=\op{Re}(iz_1\p_{z_1}),\ J=-2\op{Re}(iz_2\p_{z_2}),\\
X=\op{Re}(iz_1z_2w\p_{z_1}+\p_{z_2}+2iz_2w^2\p_w),\ Y=\op{Re}(z_1z_2w\p_{z_1}+i\p_{z_2}+2z_2w^2\p_w),\\
Z=\op{Re}(z_1w\p_{z_1}+2w^2\p_w).
\end{gather*}
This algebra is isomorphic to $\R^3\oright\mathfrak{heis}_3$, where $R$ is the grading element,
$S$ is the center and $J$ is the complex structure on the contact subspace in
 $$
\mathfrak{heis}_3=\langle X,Y,Z:[X,Y]=Z\rangle.
 $$
We have $[R,X]=X$, $[R,Y]=Y$, $[R,Z]=2Z$, $[J,X]=Y$, $[J,Y]=-X$.
 \end{remark}

\subsection{Group version of the main results}\label{S34}

Let us first note that the situations in (ii) can correspond to the existence of Levi-degenerate points as
in Theorem \ref{main1} (the models are in Example \ref{Ex6}), but the dimension can also drop by
purely topological reasons, reducing the pseudo-unitary group to its subgroup. For instance,
removing from hyperquadric \eqref{QK} a subspace $L^{s-1}$ of dimension $(s-1)$ reduces $PSU(p,q)$
to its maximal parabolic subgroup $P_{s,n-s+2}$.

The global infinitesimal automorphisms are un-altered by this removal of $L^{s-1}$, but some of the vector
fields from $\mathfrak{s}(M)$ become incomplete resulting in reduction of $G$. This is the only global effect
and it is manifested in a remarkably short proof of Theorem \ref{main3} given below.
In fact, it is a simpler statement than that for the global infinitesimal automorphisms since
realization, indicated in the previous paragraph, follows from the very definition of the
parabolic subgroup as the stabilizer of a linear subspace in the projective version of the flat model
and does not appeal to blow-ups.

\medskip

 \begin{proof}[Proof of Theorem \ref{main3}]
Let $\mathfrak{s}$ be the infinitesimal automorphism algebra of $M$ and $\g=\hbox{Lie}(G)$ the Lie algebra of $G$.
Because $\g\subset\mathfrak{s}$ the assumption of case (i) in Theorem \ref{main3} implies
the assumption of case (i) in Theorem \ref{main1} and consequently the implications align.

Consider now case (ii) in Theorem \ref{main3}. If $\g=\mathfrak{s}$ then the implications align again and we are done.
Otherwise $\dim\mathfrak{s}>\dim\g$ and this implies, by Theorem \ref{subalgmaxdim}, that $\dim\mathfrak{s}=n^2+4n+3$,
so we are under the assumption of case (i) in Theorem \ref{main1}, which yields sphericity of $M$ everywhere.
 \end{proof}

\section{Models with large symmetry}\label{S4}

We will now elaborate on CR-hypersurfaces with submaximal symmetry dimension.
First we exhibit a countably many non-equivalent models with the symmetry algebra
being the first parabolic subalgebra ${\mathfrak p}_{1,n+1}$.

Then we realize in two non-equivalent ways all maximal parabolic subalgebras proving Theorem \ref{main2}.
In particular, for $n=2$ we obtain an example of a CR-hypersurface with $\dim\mathfrak{s}(M)=n^2+2n+3=11$
that is more elementary compared to those discussed in \cite{IK2}.
For $n=4$ we get an example of a CR-hypersurface with $\dim\mathfrak{s}(M)=n^2+2n+2=26$;
its algebra $\dim\mathfrak{s}(M)=\mathfrak{p}_3\subset\mathfrak{su}(3,3)$ yields yet another model
with symmetry of the same dimension as $\mathfrak{p}_{1,5}$.

Finally we show other means to produce models with large symmetry: iterated blow-ups and ramified coverings.
In fact, both the series of examples in Section \ref{S41} and those from \cite{IK2} can be seen
as a combination of a blow-up and a ramified covering.

\subsection{A series of different realizations of ${\mathfrak p}_{1,n+1}$}\label{S41}

Fix $n\ge 1$ and $1\le p\le q$ with $p+q=n+2$, and set $(\bar{p},\bar{q})=(p-1,q-1)$. The parabolic subalgebra $\g=\mathfrak{p}_{1,n+1}\subset\mathfrak{su}(p,q)$, which has a $2$-grading $\g=\g_0\oplus\g_1\oplus\g_2$, is abstractly isomorphic to $\g_0\oright\g_+$, where $\g_0=\mathfrak{su}(\bar{p},\bar{q})\oplus\R^2$
and $\g_+=\g_1\oplus\g_2=\C^{\bar{p},\bar{q}}\oright\R$ is the Heisenberg algebra of dimension $2n+1$.

For every $m\in\N$ and $\varepsilon=\pm 1$ consider the real-analytic hypersurface $M_{m,\varepsilon}$ given in coordinates $z_1,\dots,z_n$, $w=u+iv$ in $\C^{n+1}$ by
 \begin{equation}
v=\varepsilon u\tan\left(\frac{1}{2m}\arcsin(\|z\|^2)\right),\quad \|z\|<1,\label{exampleMm}
 \end{equation}
where
 $$
\|z\|^2=\sum_{j=1}^n\sigma_j|z_j|^2=|z_1|^2+\dots+|z_{p-1}|^2-|z_p|^2-\dots-|z_n|^2
 $$
is the standard Hermitian form of signature $(\bar{p},\bar{q})$.
Here $\sigma_j=+1$ for $1\leq j\leq\bar{p}$ and $\sigma_{j}=-1$ for $p\leq j\leq n$
(notice that $\sigma_j=-1$ for all $j$ in the Levi-definite case).

For $n=1$ this hypersurface was introduced in \cite{B2} and also appeared in \cite{KL}.
Clearly, $M_{m,\varepsilon}$ contains the complex hypersurface 
$\mathfrak{S}_{m,\varepsilon}=\{\|z\|<1, w=0\}=M_{m,\varepsilon}\cap\{u=0\}$ and is Levi-nondegenerate with signature 
$(\bar p,\bar q)$ away from $\mathfrak{S}_{m,\varepsilon}$. The complement 
$M_{m,\varepsilon}\setminus\mathfrak{S}_{m,\varepsilon}$ has exactly two
connected components; they are defined by the sign of $u$.
The hypersurface $M_{m,\varepsilon}$ is not minimal, hence not of finite type (in the sense of Kohn and Bloom-Graham)
at any point of $\mathfrak{S}_{m,\varepsilon}$ (see \cite[\S 1.5]{BER}).

We now observe that every point $(z,w)\in M_{m,\varepsilon}$ satisfies the equation
 \begin{equation}
\hbox{Im}(w^{2m})\,\sqrt{1-\|z\|^4}=\varepsilon\,\hbox{Re}(w^{2m})\,\|z\|^2.\label{neweq}
 \end{equation}
In fact, for every value of $\varepsilon$, equation (\ref{neweq}) describes $2m$ pairwise CR-equivalent smooth hypersurfaces, with (\ref{exampleMm}) being one of them. The other hypersurfaces are obtained from (\ref{exampleMm}) by multiplying $w$ by a root of order $2m$ of either 1 or -1. One obtains $m$ hypersurfaces from the roots of 1 and the other $m$ ones from the roots of -1 (notice that two opposite roots lead to the same equation). All these hypersurfaces intersect along $\{w=0\}$. For example, when $m=1$ the set described by equation (\ref{neweq}) is the union of the following two smooth hypersurfaces:
 $$
\displaystyle
v=\varepsilon u\tan\left(\frac{1}{2}\arcsin(\|z\|^2)\right)\ \text{ and }\
u=-\varepsilon v\tan\left(\frac{1}{2}\arcsin(\|z\|^2)\right),\ \ \|z\|<1.
 $$

Each of the $2m$ hypersurfaces given by (\ref{neweq}) is spherical away from $\mathfrak{S}_{m,\varepsilon}$. 
Indeed, fix a point $(z_0,w_0)$ satisfying (\ref{neweq}) with $w_0\ne 0$. Then $\hbox{Re}(w_0^{2m})\ne 0$, and setting $\sigma=\varepsilon\,\hbox{sgn}\,\hbox{Re}(w_0^{2m})$, we see that the map
 $$
(z,w)\mapsto (zw^m,\sigma w^{2m})
 $$
transforms a neighborhood of $(z_0,w_0)$ on the relevant hypersurface to an open subset of the hyperquadric
(\ref{QK}) that we rewrite so
 \begin{equation}\label{hyperquadric}
{\mathcal Q}_{\bar p}=\{(z,w)\in\C^n\times\C: v=\|z\|^2\}.
 \end{equation}

 \begin{theorem}\label{Model0}
For every $m\in\N$ and $\varepsilon=\pm 1$ the symmetry algebra of $M_{m,\varepsilon}$ has dimension $n^2+2n+2$;
in fact one has $\mathfrak{s}=\mathfrak{p}_{1,n+1}$. Furthermore, for $m\ne k$ and any $\varepsilon,\delta\in\{-1,1\}$
neither the hypersurfaces $M_{m,\varepsilon}$ and $M_{k,\delta}$ nor their germs at the origin are equivalent by means
of a real-analytic CR-diffeomorphism. In addition, for $p\ne q$ neither the hypersurfaces $M_{m,-1}$ and $M_{m,+1}$
nor their germs at the origin are equivalent by means of a real-analytic CR-diffeomorphism.
 \end{theorem}

 \begin{proof}
It is straightforward to check that the following vector fields span the algebra $\mathfrak{s}(M_{m,\varepsilon})$,
where $1\leq j\leq n$, $j<\ell\leq n$ and summation over repeated indices is not assumed:
 \begin{gather}\label{vfrepres}
\op{Re}(w\p_w),\ \op{Re}\bigl(w^{2m}(m\xi+w\p_w)\bigr),
\vphantom{\frac{a}{a}}\notag\\
\op{Re}(z_j\p_{z_{\ell}}  - \sigma_j\sigma_{\ell} z_{\ell}\p_{z_j}),\
\op{Re}(iz_j\p_{z_{\ell}} +i\sigma_j\sigma_{\ell} z_{\ell}\p_{z_j}),\
\op{Re}(iz_j\p_{z_j}),
\label{vfrepres}\\
\op{Re}\bigl(z_jw^m(m\xi+w\p_w) +im\varepsilon\sigma_j w^m\p_{z_j}\bigr),\
\op{Re}\bigl(iz_jw^m(m\xi+w\p_w) +m\varepsilon\sigma_j w^m\p_{z_j}\bigr),
\vphantom{\frac{a}{a}}\notag
 \end{gather}
with $\xi=\sum_{j=1}^nz_j\p_{z_j}$. Furthermore, one can check that vector fields (\ref{vfrepres}) define a faithful representation of the parabolic subalgebra $\mathfrak{p}_{1,n+1}\subset\mathfrak{su}(p,q)$.

Since the surface $M_{m,\varepsilon}$ is not everywhere spherical, Theorem \ref{main1} yields the upper bound
$\dim\mathfrak{s}(M_{m,\varepsilon})\le n^2+2n+2+\delta_{2,n}$.
Moreover in case $n=2$ the equality is only attained for the parabolic subalgebra
$\mathfrak{p}_2\subset\mathfrak{su}(2,2)$, and since the other parabolic $\mathfrak{p}_{1,3}$
does not embed into $\mathfrak{p}_2$, we conclude that in fact $\dim\mathfrak{s}(M_{m,\varepsilon})\le n^2+2n+2$.
But since vector fields (\ref{vfrepres}) in totality $n^2+2n+2$ constitute the symmetries of the model, we conclude
the opposite inequality and hence $\mathfrak{s}(M_{m,\varepsilon})=\mathfrak{p}_{1,n+1}$.

Similarly, for any connected neighborhood $U$ of a point $x\in\mathfrak{S}_{m,\varepsilon}$ in $M_{m,\varepsilon}$ we have $\hol(U)=\mathfrak{p}_{1,n+1}$, while if $U\cap\mathfrak{S}_{m,\varepsilon}=\emptyset$ we get $\hol(U)=\mathfrak{su}(p,q)$.

Next, formulas (\ref{vfrepres}) show that all elements of $\mathfrak{s}(M_{m,\varepsilon})$ vanish precisely at the origin. Hence, if a real-analytic CR-diffeomorphism $F$ establishes equivalence between $M_{m,\varepsilon}$ and $M_{k,\delta}$, we have $F(0)=0$. Observe now that the highest order of the vanishing of a vector field in the algebra $\mathfrak{s}(M_{m,\varepsilon})$ at the origin is $2m+1$, and this number must be preserved by $F$. This shows that $m=k$. The same argument yields the nonequivalence of the germs of $M_{m,\varepsilon}$ and $M_{k,\delta}$ at the origin by means of a real-analytic CR-diffeomorphism unless $m=k$.

Further, if a real-analytic CR-diffeomorphism $F$ establishes equivalence between $M_{m,-1}$ and $M_{m,+1}$, we have
again $F(0)=0$. Since $F$ holomorphically extends to a neighbourhood of the origin, let us write it as
 $$
(z,w)\mapsto (f(z,w),g(z,w)),
 $$
with
 $$
f(z,w)=Az+Bw+\cdots,\,\,g(z,w)=Cw+\cdots,
 $$
where dots denote higher-order terms and $A,B,C$ are complex matrices of sizes
$n\times n$, $n\times 1$, $1\times 1$, respectively (note that $g$ contains no linear terms in $z$
because $F$ must preserve 
$\mathfrak{S}_{m,\varepsilon}$). The condition that $F$ maps $M_{m,-1}$ to $M_{m,+1}$ is written as the identity
 \begin{equation*}
\left.\left[
\op{Im}\bigl(g(z,w)\bigr)-\op{Re}\bigl(g(z,w)\bigr)\cdot\tan\left(\frac{1}{2m}\arcsin(\|f(z,w)\|^2)\right)
\right]\right|_{w=u- i u\tan(\frac{1}{2m}\arcsin(\|z\|^2))}=0.
 \end{equation*}
The terms linear in $u$ yield $\Im C=0$. Then $\Re C\ne 0$ and the next terms in decomposition of
the above identity imply
 $$
\|z\|^2=-\|Az\|^2.
 $$
Since signature is an invariant of the quadric, for $p\ne q$ this condition is impossible. The same argument yields the nonequivalence of the germs of $M_{m,-1}$ and $M_{m,+1}$ at the origin by means of a real-analytic CR-diffeomorphism.
 \end{proof}

\begin{remark}\label{oaramvarepsilon} \rm
The last statement of Theorem \ref{Model0} does not hold for $p=q$ (just interchange the groups of variables $z_1,\dots,z_{n/2}$ and $z_{n/2+1},\dots,z_n$). The invariance of the pair $(m,\varepsilon)$ for $n=1$ was claimed in \cite{B2}.
\end{remark}

\subsection{Realization of maximal parabolics}\label{S42}

We will now construct realizations of all the maximal parabolic subalgebras $\mathfrak{p}_{s,n-s+2}$, $1\le s\le p$.
Finding such realizations is interesting in its own right, as this adds up to
the study of symmetry of polynomial CR models, cf.\ \cite{B1,KMZ,KM}.

The first model has been already introduced in Example \ref{Ex5}, see equation \eqref{Eq-mod1}
applicable to all $1\leq s\leq\frac{n}2+1$.
It is a blow-up \eqref{pi-mod1} of the hyperquadric ${\mathcal Q}_{\bar p}$ 
along the subspace $L$ given by \eqref{L-mod1}. Let us denote this model $M^s_{\rm I}$.

Its locus of Levi degeneracy is a complex submanifold ${\mathfrak S}^s_{\rm I}$ of real dimension $2n$.
Indeed, ${\mathfrak S}^s_{\rm I}$ is an open subset of the hyperplane $\{w=0\}$ (coincides with it for $s=1$).

The second model is applicable for $1<s<\frac{n}2+1$, i.e.\ $k=s-1\in(0,\frac{n}2)$.
It is a blow up along the following subspace in $\C^n(z)\times\C(w)$ of dimension $n-k$:
 \begin{equation}\label{L-mod2}
L=\{(z,w)\in\C^{n+1}:z_j=0\ (1\leq j\leq k),\ w=0\}.
 \end{equation}
An important difference between \eqref{L-mod1} and \eqref{L-mod2} is that the latter $L$ is not contained in
${\mathcal Q}_{\bar p}$, so the blow-up happen along the real-analytic subvariety
$L'=L\cap{\mathcal Q}_{\bar p}$.
An open subset of $\op{Bl}_L{\mathcal Q}_{\bar p}$ embeds into the hypersurface
$S\subset\C^n(z)\times\C(w)$ given by
 \begin{equation}\label{Eq-mod2}
\op{Im}(w)=\sum_{j=1}^{k}\left(z_jw\bar z_{k+j}+z_{k+j}\bar w\bar z_j\right)+\|z'\|^2
 \end{equation}
($\|z'\|^2$ has the same meaning as in Example \ref{Ex5}) with the projection given by
 \begin{equation}\label{pi-mod2}
\pi_L(z_1,\dots,z_n,w)=(z_1w,\dots,z_kw, z_{k+1},\dots,z_{2k},z_{2k+1},\dots,z_n,w).
 \end{equation}
The hypersurface $S$ contains the real-analytic subvariety ${\mathfrak S}'=\{\|z'\|=0,\,w=0\}=\pi_L^{-1}(L')$,
and for every $x\in L'$ the fiber $\pi_L^{-1}(x)$ is a $k$-dimensional vector subspace of $\C^{n+1}$.
This subvariety has real dimension
 $$
\dim {\mathfrak S}'=\left\{\begin{array}{cl}
2n-1 & \hbox{if $k<\bar p$,}\\ \vspace{-0.3cm}\\
4k<2n & \hbox{if $k=\bar p$.}
\end{array}\right.
 $$
Excluding the singular locus
 $$
\Sigma_{S}=\Bigl\{\sum_{j=1}^{k}z_j\bar z_{j+k}=-\frac{i}{2},\ w=0,\ z_{\ell}=0,\ \ell=2k+1,\dots,n\Bigr\}
 $$
from $S$, we obtain our second model denoted $M^s_{\rm II}$.
Its Levi-degeneracy locus is an open subset ${\mathfrak S}^s_{\rm II}$ of ${\mathfrak S}'$.

 \begin{theorem}\label{realizgenparabolics}
Both models $M^s_{\rm I}$ {\rm(}$1\leq s\leq\frac{n}2+1${\rm)} and $M^s_{\rm II}$ {\rm(}$1<s<\frac{n}2+1${\rm)}
have symmetry algebra $\mathfrak{s}=\mathfrak{p}_{s,n-s+2}$. They are neither globally CR equivalent
nor locally equivalent near the Levi degeneracy locus.
 \end{theorem}

Of course, this assertion implies Theorem \ref{main2}. Note that complementarity of dimensions of $L$
in both cases ($s-1$ and $n-s+1$) reflects certain duality and it gives light to the fact that
both surfaces have the same symmetry algebra.

\medskip

 \begin{proof}
The symmetries of both models are obtained by straightforward but very demanding computations
(involving many {\tt Maple} experiments). For \eqref{Eq-mod1}, denoting
 $$
\zeta=\sum_{j=2k+1}^nz_j\partial_{z_j}-w\partial_w,\,\
\xi=\sum_{a=1}^kz_a\partial_{z_a}+\zeta,\,\
\eta=\sum_{a=1}^kz_{a+k}\partial_{z_{a+k}}+w\partial_w,
 $$
the following are the generators of $\mathfrak{s}(M^s_{\rm I})$ with indices in the range 
$1\le a,b\le k$, $a<c\le k$, $2k+1\le j\le n$, $j<\ell\leq n$:
 \begin{gather}
\Re(\partial_{z_a}+2iz_{a+k}\eta),\
\Re(i\partial_{z_a}+2z_{a+k}\eta),\
\Re(\partial_{z_{a+k}}-2iz_a\xi),
\Re(i\partial_{z_{a+k}}-2z_a\xi),\
 \vphantom{\frac{a}{a}}\notag\\
\Re\bigl(w(\partial_{z_{a+k}}+2iz_a\eta)\bigr),\
\Re\bigl(w(i\partial_{z_{a+k}}+2z_a\eta)\bigr),\
\Re(w\eta),\ \Re(\zeta-w\partial_{w}),
 \notag\\
\Re(iz_aw\partial_{z_{a+k}}),\
\Re(z_a\partial_{z_b}-z_{b+k}\partial_{z_{a+k}}),\
\Re(iz_a\partial_{z_b}+iz_{b+k}\partial_{z_{a+k}}),
 \vphantom{\frac{a}{a}}\notag\\
\Re\bigl(w(z_a\partial_{z_{c+k}}-z_c\partial_{z_{a+k}})\bigr),\
\Re\bigl(w(iz_a\partial_{z_{c+k}}+iz_c\partial_{z_{a+k}})\bigr),
 \label{EQgroup-I}\\
\Re(z_a\partial_{z_j}-\sigma_jz_jw\partial_{z_{a+k}}),\
\Re(iz_a\partial_{z_j}+i\sigma_jz_jw\partial_{z_{a+k}}),
 \vphantom{\frac{a}{a}}\notag\\
\Re(z_j\partial_{z_{\ell}}-\sigma_j\sigma_{\ell}z_{\ell}\partial_{z_j}),\
\Re(iz_j\partial_{z_{\ell}}+i\sigma_j\sigma_{\ell}z_{\ell}\partial_{z_j}),\
\Re(iz_j\partial_{z_j}),
 \notag\\
\Re(\partial_{z_j}+2i\sigma_jz_jw\eta),\
\Re(i\partial_{z_j}+2\sigma_jz_jw\eta).
 \vphantom{\frac{a}{a}}\notag
 \end{gather}

Similarly, for \eqref{Eq-mod2}, if we denote
 $$
\zeta=\sum_{j=2k+1}^nz_j\partial_{z_j}+w\partial_w,\,\
\xi=\sum_{a=1}^kz_{a+k}\partial_{z_{a+k}}+\zeta,\,\
\eta=\sum_{a=1}^kz_a\partial_{z_a}-w\partial_w,
 $$
and use the same range for indices, then the following are the generators of $\mathfrak{s}(M^s_{\rm II})$:
 \begin{gather}
\Re(\partial_{z_a}+2iz_{a+k}\xi),\
\Re(i\partial_{z_a}+2z_{a+k}\xi),\
\Re(\partial_{z_{a+k}}-2iz_a\eta),\
\Re(i\partial_{z_{a+k}}-2z_a\eta),
 \vphantom{\frac{a}{a}}\notag\\
\Re(w\bigl(\partial_{z_{a+k}}+2iz_a\xi)\bigr),\
\Re(w\bigl(i\partial_{z_{a+k}}+2z_a\xi)\bigr),\
\Re(w\xi),\
\Re(\zeta+w\partial_{w}),
 \notag\\
\Re(iz_aw\partial_{z_{a+k}}),\
\Re(z_a\partial_{z_b}-z_{b+k}\partial_{z_{a+k}}),\
\Re(iz_a\partial_{z_b}+iz_{b+k}\partial_{z_{a+k}}),
 \vphantom{\frac{a}{a}}\notag\\
\Re(w\bigl(z_a\partial_{z_{c+k}}-z_c\partial_{z_{a+k}})\bigr),\
\Re(w\bigl(iz_a\partial_{z_{c+k}}+iz_c\partial_{z_{a+k}})\bigr),
 \label{EQgroup-II}\\
\Re(z_j\partial_{z_{a+k}}-\sigma_jz_aw\partial_{z_j}),\
\Re(iz_j\partial_{z_{a+k}}+i\sigma_jz_aw\partial_{z_j}),
 \vphantom{\frac{a}{a}}\notag\\
\Re(z_j\partial_{z_{\ell}}-\sigma_j\sigma_{\ell}z_{\ell}\partial_{z_j}),\
\Re(iz_j\partial_{z_{\ell}}+i\sigma_j\sigma_{\ell}z_{\ell}\partial_{z_j}),\
\Re(iz_j\partial_{z_j}),
 \notag\\
\Re(w\partial_{z_j}+2i\sigma_jz_j\xi),\
\Re(iw\partial_{z_j}+2\sigma_jz_j\xi).
 \vphantom{\frac{a}{a}}\notag
 \end{gather}

Finally note that the CR-manifolds $M^s_{\rm I}$ and $M^s_{\rm II}$ for $1<s<\frac{n}2+1$
are not equivalent even by means of a smooth CR-diffeomorphism. Indeed, any such diffeomorphism must preserve
the points of Levi-degeneracy and therefore map ${\mathfrak S}^s_{\rm I}$ onto ${\mathfrak S}^s_{\rm II}$,
which is impossible since they have different dimensions. Similarly, for any point $x'\in M^s_{\rm I}$ and
any point $x''\in M^s_{\rm II}$ the germs of $(M^s_{\rm I},x')$ and $(M^s_{\rm II},x'')$
are not equivalent by means of a smooth CR-diffeomorphism.
 \end{proof}

 \begin{remark}\label{s1new}\rm
The CR-manifold $M^1_{\rm I}$ with symmetry algebra $\mathfrak{s}=\mathfrak{p}_{1,n+1}$ is not equivalent to
any of the hypersurfaces $M_{m,\varepsilon}$ introduced in (\ref{exampleMm}) by means of a real-analytic
CR-diffeomorphism. Indeed, any such diffeomorphism must preserve the points of Levi-degeneracy and therefore
map ${\mathfrak S}^1_{\rm I}$ onto $\mathfrak{S}_{m,\varepsilon}$. On the other hand, formula \eqref{EQgroup-I}
shows that the subspace of vector fields in $\mathfrak{s}(M^1_{\rm I})$ identically vanishing
on ${\mathfrak S}^1_{\rm I}$ is 1-dimensional, whereas by formulas (\ref{vfrepres}) the subspace
of vector fields in $\mathfrak{s}(M_{m,\varepsilon})$ identically vanishing on $\mathfrak{S}_{m,\varepsilon}$
has dimension $2n+2$ for any $m\in\N$, $\varepsilon=\pm 1$.

The same argument demonstrates that for any point $x\in{\mathfrak S}^1_{\rm I}$ and any point
$y\in\mathfrak{S}_{m,\varepsilon}$ the germs of $(M^1_{\rm I},x)$ and $(M_{m,\varepsilon},y)$ are not equivalent by 
means of a real-analytic CR-diffeomorphism for all $m\in\N$, $\varepsilon=\pm 1$.
 \end{remark}

 \begin{remark}\rm
Note that the spherical surface $M^1_{\rm I}\setminus\{w=0\}$ is globally equivalent to the spherical surface
$M^1_{\rm II}\setminus\{w=0\}$: the CR-diffeomorphism is given by
 \begin{multline*}
(z_1,\dots,z_k,z_{k+1},\dots,z_{2k},z_{2k+1},\dots,z_n,w)\mapsto \\ (-z_{k+1},\dots,-z_{2k},z_1,\dots,z_k,z_{2k+1},\dots,z_n,-\tfrac1w).
 \end{multline*}
This map however does not induce a transformation of \eqref{EQgroup-I} to \eqref{EQgroup-II}.
Instead it maps the parabolic subalgebra $\mathfrak{p}_{s,n-s+2}$ in $\mathfrak{su}(p,q)$
to the opposite parabolic $\mathfrak{p}^\text{op}_{s,n-s+2}$ signifying the above-mentioned duality.

\smallskip

Let us also note that other blow-ups that mix \eqref{Eq-mod1} and \eqref{Eq-mod2} like
 $$
\op{Im}(w)=\sum_{j=1}^{k}\left(z_jw\bar z_{k+j}+z_{k+j}\bar z_j\bar w\right)
+\sum_{j=2k+1}^{r}\sigma_j|z_a|^2+|w|^2\cdot\!\!\sum_{j=r+1}^n\sigma_j|z_a|^2
 $$
for $2k+1<r<n$ have the symmetry algebra strictly smaller in dimension than $\mathfrak{p}_{s,n-s+2}$;
this can be explained by Theorem \ref{symBlUp}.
 \end{remark}

\subsection{Other examples}\label{S43}

The above blow-up procedure admits a useful modification, which we will only discuss for the case when $L$ is a point. The idea is to consider, instead of the blow-up of hyperquadric (\ref{hyperquadric}) at the origin, a \lq\lq weighted\rq\rq\, blow up of a ramified cover over it.

 \begin{examp}\rm
For $r\in\N$ and $\sigma=\pm 1$ consider the map $\psi_{r,\sigma}(z,w)=(z,\sigma w^r)$.
The ramified cover $\widetilde{\mathcal Q}_{\bar p}=\psi_{r,\sigma}^{-1}({\mathcal Q}_{\bar p})$ is the hypersurface in $\C^{n+1}$ given by the equation
 $$
\Im (w^r)=\sigma\,\|z\|^2.
 $$
Clearly, $\widetilde{\mathcal Q}_{\bar p}$ is singular if $r>1$. We will now blow up $\widetilde{\mathcal Q}_{\bar p}$ at the origin by a weighted analogue of map $\pi_L$ with $L=o$, namely, by the map
 $$
\pi_{o,m}:(z,w)\mapsto (zw^m,w),\label{simpleblowupmapk}
 $$
where $m\in\N$. The result of the blow-up is the hypersurface
$R_{r,m}=\pi_{o,m}^{-1}(\widetilde{\mathcal Q}_{\bar p})$, which is described by the equation
 \begin{equation}\label{ramifcoverring}
\Im(w^r)=\sigma|w|^{2m}\,\|z\|^2.
 \end{equation}

Fix $m\in\N$ and set $r=2m$. One can rewrite equation (\ref{ramifcoverring}) of $R_{2m,m}$ as
 $$
\hbox{Im}(w^{2m})=\sigma\sqrt{\hbox{Re}(w^{2m})^2+\hbox{Im}(w^{2m})^2}\,\|z\|^2.
 $$
Taking squares, we obtain
$\hbox{Im}(w^{2m})^2\,(1-\|z\|^4)=\hbox{Re}(w^{2m})^2\,\|z\|^4$, i.e.\
 \begin{equation}
\hbox{Im}(w^{2m})\,\sqrt{(1-\|z\|^4)}=\sigma|\hbox{Re}(w^{2m})|\cdot\|z\|^2.\label{pointsetm}
 \end{equation}
The pair of equations in (\ref{pointsetm}) for $\sigma=\pm 1$ describes the same set of points as the pair of equations in (\ref{neweq}) for $\varepsilon=\pm 1$. This set is formed by $4m$ smooth hypersurfaces all intersecting along $w=0$.
Otherwise said, the models of Theorem \ref{Model0} can be considered as a ramified covering of a weighted blow-up.
 \end{examp}

 \begin{remark}\label{tmp43}\rm
For $n=1$ the weighted blow-up $R_{2m,m}$ of a ramified cover over the hyperquadric was
considered in \cite{KL}, see, e.g., Lemmas 22, 26 therein.
 \end{remark}

Let us note that surface \eqref{ramifcoverring} for $r=1$ is an iterated blow-up of the hyperquadric
\eqref{QK}. In fact, $o=(0,0)\in\C^n(z)\times\C(w)$ is the regular point of the surface
 $$
Q_m=\{\op{Im}(w)=|w|^{2m}\,\|z\|^2\}.
 $$
Note that $Q_0=\mathcal{Q}_{\bar{p}}$, where $\bar{p}=p-1$ and $(\bar{p},\bar{q})$ is the
signature of the quadric $\|z\|^2$ ($\bar{p}+\bar{q}=n$), and that $\overline{Q_{j+1}}=\op{Bl}_oQ_j$ for all $j\ge0$.

By Theorem \ref{symBlUp} the symmetry algebra of $Q_1$ is obtained from that of $Q_0$ as stabilizer of $o$ in
$\mathfrak{su}(p,q)$, the resulting algebra of vector fields has generators \eqref{EQgroup-I} for $k=0$.
This is the reduction to the parabolic subalgebra $\mathfrak{p}_{1,n+1}$.

Again, using Theorem \ref{symBlUp} the symmetry algebra of $Q_2$ is obtained from that of $Q_1$ as stabilizer of $o$
in the algebra $\mathfrak{p}_{1,n+1}$. This is obtained by removing the vector fields in the last line of
\eqref{EQgroup-I} for $k=0$. Further blow-ups do not change the dimension (only some coefficients are being modified),
and we conclude $\dim\mathfrak{s}(Q_m)=n^2+2$.

Actually, this dimension persists for the symmetry algebra of the ramified equation.

 \begin{prop}
The symmetry algebra of \eqref{ramifcoverring} for $m>1$ is $\mathfrak{u}(\bar{p},\bar{q})\oplus\mathfrak{sol}(2)$,
where $\mathfrak{sol}(2)$ is the two-dimensional solvable Lie algebra.
 \end{prop}

 \begin{proof}
Since we already restricted the dimension from above, it is enough to indicate the generators.
They are given below with the index range $1\le\ell\le n$, $\ell<j\le n$:
 \begin{gather*}
\Re(z_{\ell}\partial_{z_j}-\sigma_{\ell}\sigma_jz_j\partial_{z_{\ell}}),\
\Re(iz_{\ell}\partial_{z_j}+i\sigma_{\ell}\sigma_jz_j\partial_{z_{\ell}}),\
\Re(iz_{\ell}\partial_{z_{\ell}}),\\
\Re\bigl(w\partial_w-(m-\tfrac{r}2)\sum_{j=1}^n\sigma_jz_j\partial_{z_j}\bigr),\
\Re\bigl(w^r(w\partial_w-(m-r)\sum_{j=1}^n\sigma_jz_j\partial_{z_j})\bigr).
 \end{gather*}
The abstract Lie algebra structure is straightforward.
 \end{proof}

Finally, let us show an example of iterated blow-up giving a surface with a non-maximal
parabolic symmetry algebra. We do it in the simplest case $n=2$ with parabolic
$\mathfrak{p}_{1,2,3}$ in $\mathfrak{su}(2,2)$ being the Borel subalgebra.

 \begin{examp}\label{p123}\rm
If we blow up the hyperquadric $\op{Im}(w)=2\Re(z_1\bar{z_2})$ in $\C^3$
along $L_1=\{z_1=0=w\}$ via $\pi_{L_1}(z_1,z_2,w)=(z_1w,z_2,w)$ we get the surface
 \begin{equation}\label{e:p2}
\op{Im}(w)=2\Re(z_1w\bar{z_2})
 \end{equation}
with the symmetry algebra $\mathfrak{p}_2\subset\mathfrak{su}(2,2)$. Blowing it up
again along $L_1=\{z_1=0=w\}$ (in new coordinates) we get the surface
 \begin{equation}\label{e:p123}
\op{Im}(w)=2\Re(z_1w^2\bar{z_2}).
 \end{equation}

 \begin{prop}
The symmetry algebra of \eqref{e:p123} is $\mathfrak{p}_{1,2,3}\subset\mathfrak{su}(2,2)$.
 \end{prop}

 \begin{proof}
The symmetry algebra has generators:
 \begin{gather*}
\Re(3z_1\p_{z_1}-z_2\p_{z_2}-2w\p_w),\
\Re\bigl(w(z_1\p_{z_1}-z_2\p_{z_2}-w\p_w)\bigr),\\
\Re\bigl(\p_{z_2}-2iz_1w(2z_1\p_{z_1}-w\p_w)\bigr),\ \Re\bigl(i\p_{z_2}-2z_1w(2z_1\p_{z_1}-w\p_w)\bigr),\\
\Re\bigl(w\p_{z_2}-2iz_1w^2(z_1\p_{z_1}-z_2\p_{z_2}-w\p_w)\bigr),\
\Re\bigl(iw\p_{z_2}-2z_1w^2(z_1\p_{z_1}-z_2\p_{z_2}-w\p_w)\bigr),\\
\Re(z_1\p_{z_1}-z_2\p_{z_2}),\ \Re(iz_1\p_{z_1}+iz_2\p_{z_2}),\
\Re(iz_1w^2\p_{z_2}).
 \end{gather*}
They satisfy the structure equations of the Borel subalgebra in $\mathfrak{su}(2,2)$.
In fact, the stabilizer of $L_1$ in 11-dimensional symmetry algebra of \eqref{e:p2} is
the indicated 9-dimensional subalgebra $\mathfrak{p}_{1,2,3}\subset\mathfrak{p}_2$,
in accordance with Theorem \ref{symBlUp}.
 \end{proof}
 \end{examp}

 \begin{examp}\label{iterative}\rm
We remark that (iteratively) blowing up surface $M$ \eqref{e:p2} again along $L_2=\{z_2=0=w\}$
via $\pi_{L_2}(z_1,z_2,w)=(z_1,z_2w,w)$ we get the surface
 \begin{equation}\label{e:p13}
\op{Im}(w)=2\Re(z_1\bar{z_2})\,|w|^2
 \end{equation}
with the symmetry algebra $\mathfrak{p}_{1,3}\subset\mathfrak{su}(2,2)$.
This seems to contradict Theorem \ref{symBlUp} because $\mathfrak{p}_{1,3}$ does not embed into $\mathfrak{p}_2$.

The explanation is as follows. The surface \eqref{e:p13} is only an open dense part of the blow-up of
\eqref{e:p2} along $L_2$. The projection $\pi_{L_2}$ is not epimorphic on $M$, the subset
$\{(z_1,z_2,0):z_2\neq0\}$ is not in the image of $\pi_{L_2}$ restricted to \eqref{e:p13}.
This implies that the symmetry of $M$ fixing this subset and $L_1$ (as per definition)
also fixes their intersection, i.e.\ the point $o=(0,0,0)$ leading to the symmetry algebra $\mathfrak{p}_{1,3}$.

The blow-up $\op{Bl}_{L_2}M$ of the surface $M$ \eqref{e:p2} is obtained from
 $$
M'=\{\op{Im}(w')=2\Re(z'_1\bar{z}'_2)\,|w'|^2\}\ \ \text{ and }\ \
M''=\{\op{Im}(z_2''w'')=2\Re(z''_1w'')\,|z_2''|^2\},
 $$
with singular lines $\{(z''_1,0,0)\}$ and $\{(z''_1,\frac{i}{2\bar{z}_1''},0):z_1''\neq0\}$ removed from $M''$,
by gluing them via $\varphi(z_1',z_2',w')=(z_1'',z_2'',w'')=(z_1',z_2'w',1/z_2')$. And this
bigger surface gives reduction of the symmetry to an 8-dimensional subalgebra
of both $\mathfrak{p}_{1,3}$ and $\mathfrak{p}_2$.
 \end{examp}

\section{Conclusion}\label{S5}

Let us outline a possible generalization and formulate some open problems.

\subsection{On generalization of the main result}\label{S51}

Motivated by results in the present paper and a series of preceding works in complex analysis,
we formulate the following claim, generalizing Conjecture~\ref{beloshapka}.

 \begin{conjecture}\label{new-conjecture}
Symmetry of any real-analytic connected holomorphically nondegenerate CR-hypersurface $M$ of CR-dimension $n$
satisfies $\dim\mathfrak{s}(M)\le n^2+4n+3$, with the maximal value attained only if $M$ is
everywhere spherical. Otherwise $\dim\mathfrak{s}(M)\le n^2+2n+2+\delta_{2,n}$, with the maximal value
attained only if on a dense open set $M$ is spherical and of fixed signature of the Levi form.
 \end{conjecture}

Let us support this claim. It holds for $1\leq n\leq 2$, and also for larger $n$, provided $M$ is 
Levi nondegenerate somewhere. For the case $n=2$ we utilized in \cite{IK2} the following fact:
a real-analytic connected holomorphically nondegenerate CR-hypersurface of dimension 5 with everywhere
degenerate Levi form is generically 2-nondegenerate. By appealing to the main result of \cite{IZ},
this allowed to estimate the dimension of the symmetry algebra in everywhere Levi-degenerate case by 10,
see also \cite{MS,MP}.

For $n\ge 3$ some partial results generalizing this have been obtained in the literature.
CR hypersurfaces that are 1-degenerate and 2-nondegenerate in the sense of Freeman
with a certain additional condition were investigated in
\cite{P} for $n=3$ and in \cite{PZ} for general $n$. The upper bound on symmetry achieved in those
references confirms our conjecture. Also in \cite{Sa} all Levi degenerate homogeneous 7-dimensional
CR hypersurfaces ($n=3$) were classified. Again, the results align with Conjecture \ref{new-conjecture}.

We expect that elaboration upon Cartan and Tanaka theories in the spirit of \cite{K1} can
provide effective bounds on local symmetry important for this claim.
Global topological behavior of $M$ results in passing from a local algebra to a subalgebra and,
by the results of Section \ref{S31}, this cannot change the submaximal dimension bound.

\subsection{On models with large symmetry}\label{S52}

Realizations of many symmetry algebras remain beyond the scope of this paper. For instance, we conjecture that
non-maximal parabolic subalgebras of the pseudo-unitary algebras can also be realized as symmetries of polynomial 
CR models. Realization of other maximal subalgebras, discussed in the proof of Theorem \ref{subalgmaxdim},
is important too.

Large symmetry algebras can also be obtained via intersection of maximal subalgebras. For instance, blow-up of
the hyperquadric ${\mathcal Q}_{\bar p}$ at two different points reduces the symmetry algebra $\mathfrak{su}(p,q)$ 
to the intersection of two conjugated parabolics $\mathfrak{p}_{1,n+1}$ that vary in dimension depending on position of the points.

In \cite{IK2} we presented a series of examples of 5-dimensional CR manifolds with symmetry dimension 
$\leq11$. In particular, for $\dim\mathfrak{s}=9$ we borrowed the following example from \cite{KM}: 
$M^5=\{\Im(w)=|z_1|^2+|z_2|^4\}$. We computed that its symmetry algebra is 
$\mathfrak{s}(M^5)=\mathfrak{u}(1,2)$, a reductive maximal subalgebra in $\mathfrak{su}(2,2)$. 
Note that the symmetry algebra $\mathfrak{p}_{1,2,3}$, constructed in Example \ref{p123}, 
is also of dimension 9. 

The same problem is interesting for the automorphism group. For $n=2$
consider the lens space ${\mathcal L}_m=S^5/\Z_m$, $m>1$, where $\Z_m\subset U(1)$ acts on
the unit sphere $S^5\subset\C^3$ by complex multiplication. By \cite[p.~37]{I1} the Lie group
$\op{Hol}({\mathcal L}_m)$ is $U(3)/\Z_m$, again of dimension 9.
Note that ${\mathcal L}_m$ is everywhere spherical falling into part (ii) of Theorem~\ref{main3}.

Finally, provided the Levi nondegeneracy locus is nonempty for $n>2$, the models of 
symmetry dimension $D_{\hbox{\tiny\rm max}}$ are spherical. Classification of CR-hypersurfaces
with symmetry dimension $D_{\hbox{\tiny\rm smax}}$ is not fully solved even for $n=1$.
The real difficulties show in the construction of the models in \cite{IK2}. 
The approach taken in this paper suggests more tractable problems: 
Which weighted blow-ups and ramified coverings of the hyperquadric lead to the models with submaximal symmetry dimension?
Can these be classified? We hope these directions show fruitful in the future.

\end{document}